\documentclass[runningheads, a4paper]{llncs}

\usepackage[utf8]{inputenc}
\usepackage[T1]{fontenc}
\usepackage[british]{babel}

\usepackage{textcomp}
\usepackage{scrhack}
\usepackage{xspace}
\usepackage{mparhack}
\usepackage{fixltx2e}

\usepackage{amsmath}
\usepackage[warning, all]{onlyamsmath}\AtBeginDocument{\catcode`\$=3}
\usepackage{amssymb}
\usepackage{stmaryrd}
\usepackage{eufrak}
\usepackage{mathtools}
\usepackage{mathdots}
\usepackage{gensymb}
\allowdisplaybreaks

\usepackage{graphicx}
\usepackage{scalerel}
\usepackage{csquotes}
\usepackage{url}

\spnewtheorem{main-theorem}[theorem]{Main Theorem}{\bfseries}{\itshape}

\urldef{\mails}\path|simon.wacker@kit.edu|
\newcommand{\keywords}[1]{\par\addvspace\baselineskip\noindent\keywordname\enspace\ignorespaces#1}

\newcommand*{\define}[1]{\emph{#1}}

\providecommand\iff{\DOTSB\;\Longleftrightarrow\;}
\providecommand\implies{\DOTSB\;\Longrightarrow\;}

\DeclarePairedDelimiter\set{\{}{\}} % http://tex.stackexchange.com/questions/5502/how-to-get-a-mid-binary-relation-that-grows
\DeclarePairedDelimiter\family{\{}{\}}

\mathchardef\breakingcomma\mathcode`\,
{\catcode`,=\active
  \gdef,{\breakingcomma\discretionary{}{}{}}
}
\newcommand*\ntuple[1]{\lparen\mathcode`\,=\string"8000 #1\rparen}

\newcommand*{\naturals}{\mathbb{N}}

\DeclareMathOperator{\Exists}{\exists}
\DeclareMathOperator{\ForEach}{\forall}
\newcommand*\Holds{:}
\newcommand*\SuchThat{:}

\newcommand*\leftaction{\triangleright}
\newcommand*\inducedleftaction{\mathbin{\protect\scalerel*{\blacktriangleright}{\triangleright}}}
\newcommand*\rightsemiaction{\mathbin{\protect\scalerel*{\trianglelefteqslant}{\rhd}}}

\newcommand*{\transpose}{\intercal} % http://tex.stackexchange.com/questions/30619/what-is-the-best-symbol-for-vector-matrix-transpose

\newcommand*\suchthat{\mid} % \set{... \suchthat ...} % \mathrel{}\middle|\mathrel{} http://tex.stackexchange.com/questions/448/how-to-automatically-resize-the-vertical-bar-in-a-set-comprehension
\newcommand*\from{\colon} % f \from M \to N % \mathpunct{:} http://tex.stackexchange.com/questions/37789/using-colon-or-in-formulas
\newcommand*\after{\circ}
\newcommand*\quotient{\slash}
\renewcommand*{\restriction}{\mathord{\upharpoonright}}

\newcommand*{\blank}{\mathord{\_}} % \cdot http://tex.stackexchange.com/questions/47060/placeholder-for-variable-as-in-fx

\begin{document}

  \mainmatter

  \title{Cellular Automata on Group Sets and the Uniform Curtis-Hedlund-Lyndon Theorem}
  \titlerunning{Cellular Automata on Group Sets}
  \authorrunning{Cellular Automata on Group Sets}

  \author{Simon Wacker}

  \institute{%
    Karlsruhe Institute of Technology\\
    \mails\\
    \url{http://www.kit.edu}%
  }

  \maketitle

  \begin{abstract}
    We introduce cellular automata whose cell spaces are left homogeneous spaces and prove a uniform as well as a topological variant of the Curtis-Hedlund-Lyndon theorem. Examples of left homogeneous spaces are spheres, Euclidean spaces, as well as hyperbolic spaces acted on by isometries; vertex-transitive graphs, in particular, Cayley graphs, acted on by automorphisms; groups acting on themselves by multiplication; and integer lattices acted on by translations. % A left homogeneous space consists of a set $M$, the cells, and a group $G$, the symmetries, that acts transitively on $M$ on the left by a map $\leftaction \colon G \times M \to M$ that respects the group operation. The role of translations in classical cellular automata is taken by the maps $g \leftaction \blank$, $g \in G$.
  \end{abstract}

  In the first chapter of the monograph \enquote{Cellular Automata and Groups}\cite{ceccherini-silberstein:coornaert:2010}, Tullio Ceccherini-Silberstein and Michel Coornaert develop the theory of cellular automata whose cell spaces are groups. Examples of groups are abound: The integer lattices and Euclidean spaces with addition (translation), the one-dimensional unit sphere embedded in the complex plane with complex multiplication (rotation), and the vertices of a Cayley graph with the group structure it encodes (graph automorphisms).

  Yet, there are many structured sets that do not admit a structure-preserving group structure. For example: Each Euclidean $n$-sphere, except for the zero-, one-, and three-dimensional, does not admit a topological group structure; the Petersen graph is not a Cayley graph and does thus not admit an edge-preserving group structure on its vertices. However, these structured sets can be acted on by subgroups of their automorphism group by function application. For example Euclidean $n$-spheres can be acted on by rotations about the origin and graphs can be acted on by edge-preserving permutations of their vertices.

  Moreover, there are structured groups that have more symmetries than can be expressed with the group structure. The integer lattices and the Euclidean spaces with addition, for example, are groups, but addition expresses only their translational symmetries and not their rotational and reflectional ones. Though, they can be acted on by arbitrary subgroups of their symmetry groups, like the ones generated by translations and rotations.

  The general notion that encompasses these structure-preserving actions is that of a group set, that is, a set that is acted on by a group. A group set $M$ acted on by $G$ such that for each tuple $(m, m') \in M \times M$ there is a symmetry $g \in G$ that transports $m$ to $m'$ is called \emph{left homogeneous space} and the action of $G$ on $M$ is said to be \emph{transitive}. In particular, groups are left homogeneous spaces --- they act on themselves on the left by multiplication.

  In this paper, we develop the theory of cellular automata whose cell spaces are left homogeneous spaces which culminates in the proof of a uniform and topological variant of a famous theorem by Morton Landers Curtis, Gustav Arnold Hedlund, and Roger Conant Lyndon from 1969, see the paper \enquote{Endomorphisms and automorphisms of the shift dynamical system}\cite{hedlund:1969}. The development of this theory is greatly inspired by \cite{ceccherini-silberstein:coornaert:2010}.

  These cellular automata are defined so that their global transition functions are equivariant under the induced group action on global configurations. Depending on the choice of the cell space, these actions may be plain translations but also rotations and reflections. Exemplary for the first case are integer lattices that are acted on by translations; and for the second case Euclidean $n$-spheres that are acted on by rotations, but also the two-dimensional integer lattice that is acted on by the group generated by translations and the rotation by 90\degree.

  S{\'{e}}bastien Moriceau defines and studies a more restricted notion of cellular automata over group sets in his paper \enquote{Cellular Automata on a $G$-Set}\cite{moriceau:2011}. He requires sets of states and neighbourhoods to be finite. His automata are the global transition functions of, what we call, semi-cellular automata with finite set of states and finite essential neighbourhood.

  His automata obtain the next state of a cell by shifting the global configuration such that the cell is moved to the origin, restricting that configuration to the neighbourhood of the origin, and applying the local transition function to that local configuration. Our automata obtain the next state of a cell by determining the neighbours of the cell, observing the states of that neighbours, and applying the local transition function to that observed local configuration. The results are the same but the viewpoints are different, which manifests itself in proofs and constructions.

  To determine the neighbourhood of a cell we let the relative neighbourhood semi-act on the right on the cell. That right semi-action is to the left group action what right multiplication is to the corresponding left group multiplication. Many properties of cellular automata are a consequence of the interplay between properties of that semi-action, shifts of global configurations, and rotations of local configurations. That semi-action also plays an important role in our definition of right amenable left group sets, see \cite{wacker:amenable:2016}, for which the Garden of Eden theorem holds, see \cite{wacker:garden:2016}, which states that each cellular automaton is surjective if and only if it is pre-injective. For example finitely right generated left homogeneous spaces of sub-exponential growth are right amenable, in particular, quotients of finitely generated groups of sub-exponential growth by finite subgroups acted on by left multiplication.

  In Sect.~\ref{sec:actions} we introduce left group actions and our prime example, which illustrates phenomena that cannot be encountered in groups acting on themselves on the left by multiplication. In Sect.~\ref{sec:semi-action} we introduce coordinate systems, cell spaces, and right quotient set semi-actions that are induced by left group actions and coordinate systems. In Sect.~\ref{sec:automata} we introduce semi-cellular and cellular automata. In Sect.~\ref{sec:invariances} we show that a global transition function does not depend on the choice of coordinate system, is equivariant under the induced left group action on global configurations, is determined by its behaviour in the origin, and that the composition of two global transition functions is a global transition function. In Sect.~\ref{sec:characterisation} we prove a uniform and a topological variant of the Curtis-Hedlund-Lyndon theorem, which characterise global transition functions of semi-cellular automata by uniform and topological properties respectively. And in Sect.~\ref{sec:invertibility} we characterise invertibility of semi-cellular automata.

  %----------------------------------------------------------------------------------------
  \section{Left Group Actions}
  \label{sec:actions}
  %----------------------------------------------------------------------------------------

  \begin{definition}
    Let $M$ be a set, let $G$ be a group, let $\leftaction$ be a map from $G \times M$ to $M$, and let $e_G$ be the neutral element of $G$. The map $\leftaction$ is called \define{left group action of $G$ on $M$}, the group $G$ is said to \define{act on $M$ on the left by $\leftaction$}, and the triple $\ntuple{M, G, \leftaction}$ is called \define{left group set} if and only if
    \begin{gather*}
      \ForEach m \in M \Holds e_G \leftaction m = m,\\
      \ForEach m \in M \ForEach g \in G \ForEach g' \in G \Holds g g' \leftaction m = g \leftaction (g' \leftaction m).
    \end{gather*}
  \end{definition}

  \begin{example}
  \label{ex:group:action}
    Let $G$ be a group. It acts on itself on the left by $\leftaction$ by multiplication.
  \end{example}

  \begin{example}
  \label{ex:sphere:action}
    Let $M$ be the Euclidean unit $2$-sphere, that is, the surface of the ball of radius $1$ in $3$-dimensional Euclidean space, and let $G$ be the rotation group. The group $G$ acts on $M$ on the left by $\leftaction$ by function application, that is, by rotation about the origin.
  \end{example}

  \begin{definition}
    Let $\leftaction$ be a left group action of $G$ on $M$ and let $H$ be a subgroup of $G$. The left group action $\leftaction\restriction_{H \times M}$ of $H$ on $M$ is denoted by $\leftaction_H$.
  \end{definition}

  \begin{definition}
    Let $\leftaction$ be a left group action of $G$ on $M$. It is called
    \begin{enumerate}
      \item \define{transitive} if and only if the set $M$ is non-empty and
            \begin{equation*}
              \ForEach m \in M \ForEach m' \in M \Exists g \in G \SuchThat g \leftaction m = m';
            \end{equation*}
      \item \define{free} if and only if
            \begin{equation*}
              \ForEach g \in G \ForEach g' \in G \Holds (\Exists m \in M \SuchThat g \leftaction m = g' \leftaction m) \implies g = g'.
            \end{equation*}
    \end{enumerate}
  \end{definition}

  \begin{example}
  \label{ex:group:transitive-and-free}
    In the situation of Example~\ref{ex:group:action}, the left group action is transitive and free.
  \end{example}

  \begin{example}
  \label{ex:sphere:transitive}
    In the situation of Example~\ref{ex:sphere:action}, the left group action is transitive but not free.
  \end{example}

  \begin{definition}
    Let $\ntuple{M, G, \leftaction}$ be a left group set. It is called \define{left homogeneous space} if and only if the left group action $\leftaction$ is transitive.
  \end{definition}

%  \begin{definition}
%    Let $\leftaction$ be a left group action of $G$ on $M$ and let $P$ be an adjective. The group $G$ is said to \define{act $P$ly on $M$ on the left by $\leftaction$} if and only if the group action $\leftaction$ is $P$.
%  \end{definition}

  \begin{definition}
    Let $\leftaction$ be a left group action of $G$ on $M$, and let $m$ and $m'$ be two elements of $M$.
    \begin{enumerate}
      \item The set $G \leftaction m = \set{g \leftaction m \suchthat g \in G}$ is called \define{orbit of $m$ under $\leftaction$}.
      \item The set $G_m = \set{g \in G \suchthat g \leftaction m = m}$ is called \define{stabiliser of $m$ under $\leftaction$}.
      \item The set $G_{m,m'} = \set{g \in G \suchthat g \leftaction m = m'}$ is called \define{transporter of $m$ to $m'$ under $\leftaction$}.
    \end{enumerate}
  \end{definition}

  \begin{example}
  \label{ex:group:orbit-stabiliser}
    In the situation of Example~\ref{ex:group:transitive-and-free}, each orbit is $G$ and each stabiliser is $\set{e_G}$. % the trivial subgroup of $G$.
  \end{example}

  \begin{example}
  \label{ex:sphere:orbit-stabiliser}
    In the situation of Example~\ref{ex:sphere:transitive}, for each point $m \in M$, its orbit is $M$ and its stabiliser is the group of rotations about the line through the origin and itself.
  \end{example}

  \begin{lemma}
  \label{lem:stabiliser-vs-transporter}
    Let $\leftaction$ be a left group action of $G$ on $M$, let $m$ and $m'$ be two elements of $M$ that have the same orbit under $\leftaction$, and let $g$ be an element of $G_{m,m'}$. Then, $G_{m'} = g G_m g^{-1}$ and $g G_m = G_{m,m'} = G_{m'} g$. \qed
  \end{lemma}

  \begin{proof}%[Lemma~\ref{lem:stabiliser-vs-transporter}]
    First, let $g' \in G_m$. Then,
    \begin{equation*}
      g g' g^{-1} \leftaction m' = g g' \leftaction m
                                 = g \leftaction m
                                 = m'.
    \end{equation*}
    Hence, $g g' g^{-1} \in G_{m'}$. In conclusion, $G_{m'} \supseteq g G_m g^{-1}$. Secondly, let $g' \in G_{m'}$. Then, as above, $g'' = g^{-1} g' g \in G_m$. Hence, $g' = g g'' g^{-1} \in g G_m g^{-1}$. In conclusion, $G_{m'} \subseteq g G_m g^{-1}$. To sum up, $G_{m'} = g G_m g^{-1}$. Moreover, $g G_m = G_{m'} g$.

    Thirdly, let $g' \in G_m$. Then, $g g' \in G_{m,m'}$. In conclusion, $g G_m \subseteq G_{m,m'}$. Lastly, let $g' \in G_{m,m'}$. Then, $g'' = g^{-1} g' \in G_m$. Hence, $g' = g g'' \in g G_m$. In conclusion, $g G_m \supseteq G_{m,m'}$. To sum up, $g G_m = G_{m,m'}$. \qed
  \end{proof}

  \begin{definition}
    Let $M$ and $M'$ be two sets, let $f$ be a map from $M$ to $M'$, and let $\leftaction$ be a left group action of $G$ on $M$. The map $f$ is called \define{$\leftaction$-invariant} if and only if
    \begin{equation*}
      \ForEach g \in G \ForEach m \in M \Holds f(g \leftaction m) = f(m).
    \end{equation*}
  \end{definition}

  \begin{definition}
    Let $M$ and $M'$ be two sets, let $f$ be a map from $M$ to $M'$, and let $\leftaction$ and $\leftaction'$ be two left group actions of $G$ on $M$ and $M'$ respectively. The map $f$ is called
%    \begin{enumerate}
%      \item
            \define{$(\leftaction, \leftaction')$-equivariant} if and only if
            \begin{equation*}
              \ForEach g \in G \ForEach m \in M \Holds f(g \leftaction m) = g \leftaction' f(m);
            \end{equation*}
%      \item
            and
            \define{$\leftaction$-equivariant} if and only if it is $(\leftaction, \leftaction')$-equivariant, $M = M'$, and $\leftaction = \leftaction'$. % TADA Unschön!
%    \end{enumerate}
  \end{definition}

  \begin{lemma}
  \label{lem:inverse-is-equivariant}
    Let $f$ be a $(\leftaction, \leftaction')$-equivariant and bijective map from $M$ to $M'$. The inverse of $f$ is $(\leftaction', \leftaction)$-equivariant
  \end{lemma}

  \begin{proof}%[Lemma~\ref{lem:inverse-is-equivariant}] % TADA Ärgerlich, dass die Gruppe G nicht im Lemma genannt wird!
    For each $g \in G$ and each $m \in M$,
    \begin{equation*}
      f^{-1}(g \leftaction m)
      = f^{-1}(g \leftaction f(f^{-1}(m)))
      = f^{-1}(f(g \leftaction f^{-1}(m)))
      = g \leftaction f^{-1}(m). \tag*{\qed}
    \end{equation*}
  \end{proof}

  \begin{lemma}
  \label{lem:induced-left-group-action-on-quotient-set}
    Let $G$ be a group and let $H$ be a subgroup of $G$. The group $G$ acts transitively on the quotient set $G \quotient H$ on the left by
    \begin{align*}
      \cdot \from G \times G \quotient H &\to     G \quotient H,\\
                               (g, g' H) &\mapsto g g' H.
    \end{align*}
  \end{lemma}

  \begin{proof}%[Lemma~\ref{lem:induced-left-group-action-on-quotient-set}]
    The map $\cdot$ is well-defined, because, for each $g \in G$, each $g_1' \in G$, and each $g_2' \in G$,
    \begin{equation*}
      g_1' H = g_2' H \iff g g_1' H = g g_2' H.
    \end{equation*}
    It is a left group action, because, for each $g' H \in G \quotient H$, % each left coset $h H \in G \quotient H$
    \begin{equation*}
      e_G \cdot g' H = g' H,
    \end{equation*}
    and, for each $g_1 \in G$, each $g_2 \in G$, and each $g' H \in G \quotient H$,
    \begin{align*}
      g_1 g_2 \cdot g' H &= g_1 g_2 g' H\\
                         &= g_1 \cdot g_2 g' H\\
                         &= g_1 \cdot (g_2 \cdot g' H).
    \end{align*}
    It is transitive, because, for each $g_1' H \in G \quotient H$ and each $g_2' H \in G \quotient H$,
    \begin{equation*}
      g_2' {g_1'}^{-1} \cdot g_1' H = g_2' H. \tag*{\qed}
    \end{equation*}
  \end{proof}

  \begin{lemma}
  \label{lem:iota-is-bijective-and-equivariant}
    Let $\leftaction$ be a transitive left group action of $G$ on $M$, let $m_0$ be an element of $M$, and let $G_0$ be the stabiliser of $m_0$ under $\leftaction$. The map
    \begin{align*}
      \iota \from M &\to     G \quotient G_0,\\
                  m &\mapsto G_{m_0, m},
    \end{align*}
    is $(\leftaction, \cdot)$-equivariant and bijective. % and its inverse is $(\cdot, \leftaction)$-equivariant.
  \end{lemma}

  \begin{proof}%[Lemma~\ref{lem:iota-is-bijective-and-equivariant}]
    For each $g \in G$ and each $m \in M$,
    \begin{equation*}
      \iota(g \leftaction m)
      = G_{m_0, g \leftaction m}
      = g \cdot G_{m_0, m}
      = g \cdot \iota(m).
    \end{equation*}
    Hence, $\iota$ is $(\leftaction, \cdot)$-equivariant. Moreover, for each $(m, m') \in M \times M$ with $m \neq m'$, we have $G_{m_0, m} \neq G_{m_0, m'}$. Thus, $\iota$ is injective. Furthermore, for each $g G_0 \in G \quotient G_0$, we have $G_{m_0, g \leftaction m_0} = g G_0$. Therefore, $\iota$ is surjective. \qed
  \end{proof}

  %----------------------------------------------------------------------------------------
  \section{Right Quotient Set Semi-Actions}
  \label{sec:semi-action}
  %----------------------------------------------------------------------------------------

  \begin{definition}
  \label{def:coordinate-system}
    Let $\mathcal{M} = \ntuple{M, G, \leftaction}$ be a left homogeneous space, let $m_0$ be an element of $M$, let $g_{m_0, m_0}$ be the neutral element of $G$, and, for each element $m \in M \setminus \set{m_0}$, let $g_{m_0, m}$ be an element of $G$ such that $g_{m_0, m} \leftaction m_0 = m$. The tuple $\mathcal{K} = \ntuple{m_0, \family{g_{m_0, m}}_{m \in M}}$ is called \define{coordinate system for $\mathcal{M}$}; the element $m_0$ is called \define{origin}; for each element $m \in M$, the element $g_{m_0, m}$ is called \define{coordinate of $m$}; for each subgroup $H$ of $G$, the stabiliser of the origin $m_0$ under $\leftaction_H$, which is $G_{m_0} \cap H$, is denoted by $H_0$.
  \end{definition}

  \begin{definition} % TADA This could be combined with the definition of coordinate systems.
  \label{def:cell-space}
    Let $\mathcal{M} = \ntuple{M, G, \leftaction}$ be a left homogeneous space and let $\mathcal{K} = \ntuple{m_0, \family{g_{m_0, m}}_{m \in M}}$ be a coordinate system for $\mathcal{M}$. The tuple $\mathcal{R} = \ntuple{\mathcal{M}, \mathcal{K}}$ is called \define{cell space}, each element $m \in M$ is called \define{cell}, and each element $g \in G$ is called \define{symmetry}.
  \end{definition} % TADA Make examples about coordinate systems examples about cell spaces?

  \begin{example}
  \label{ex:group:cell-space}
    In the situation of Example~\ref{ex:group:orbit-stabiliser}, let $m_0$ be the neutral element $e_G$ of $G$ and, for each element $m \in G$, let $g_{m_0, m}$ be the only element in $G$ such that $g_{m_0, m} m_0 = m$, namely $m$. The tuple $\mathcal{K} = \ntuple{m_0, \family{g_{m_0, m}}_{m \in M}}$ is a coordinate system of $\mathcal{M} = \ntuple{G, G, \cdot}$ and the tuple $\mathcal{R} = \ntuple{\mathcal{M}, \mathcal{K}}$ is a cell space.
  \end{example}

  \begin{example}
  \label{ex:sphere:cell-space}
    In the situation of Example~\ref{ex:sphere:orbit-stabiliser}, let $m_0$ be the north pole $(0,0,1)^\transpose$ of $M$ and, for each point $m \in M$, let $g_{m_0, m}$ be a rotation about an axis in the $(x, y)$-plane that rotates $m_0$ to $m$. Note that $g_{m_0, m_0}$ is the trivial rotation. The tuple $\mathcal{K} = \ntuple{m_0, \family{g_{m_0, m}}_{m \in M}}$ is a coordinate system of $\mathcal{M} = \ntuple{M, G, \leftaction}$ and the tuple $\mathcal{R} = \ntuple{\mathcal{M}, \mathcal{K}}$ is a cell space.
  \end{example}

  In the remainder of this section, let $\mathcal{R} = \ntuple{\ntuple{M, G, \leftaction}, \ntuple{m_0, \family{g_{m_0, m}}_{m \in M}}}$ be a cell space.

  \begin{lemma}
  \label{lem:liberation}
    The map
    \begin{align*}
      \rightsemiaction \from M \times G \quotient G_0 &\to     M,\\
                                           (m, g G_0) &\mapsto g_{m_0, m} g g_{m_0, m}^{-1} \leftaction m\ (= g_{m_0, m} g \leftaction m_0),
    \end{align*}
    is a \define{right quotient set semi-action of $G \quotient G_0$ on $M$ with defect $G_0$}, which means that, for each subgroup $H$ of $G$ such that $\set{g_{m_0, m} \suchthat m \in M} \subseteq H$,
    \begin{gather*}
      \ForEach m \in M \Holds m \rightsemiaction G_0 = m,\\
      \ForEach m \in M \ForEach h \in H \Exists h_0 \in H_0 \SuchThat \ForEach \mathfrak{g} \in G \quotient G_0 \Holds
            m \rightsemiaction h \cdot \mathfrak{g} = (m \rightsemiaction h G_0) \rightsemiaction h_0 \cdot \mathfrak{g}.
    \end{gather*} % TADA and said to be \define{induced by $\mathcal{R}$}.
  \end{lemma}

  \begin{proof}%[Lemma~\ref{lem:liberation}]
    % TADA Charakterisierung aus Lemma zeigen? Und Wohldefiniertheit?
%    The map $\rightsemiaction$ is well-defined: For each $g G_0 \in G \quotient G_0$ and each $g' G_0 \in G \quotient G_0$ with $g G_0 = g' G_0$, there is a $g_0 \in G_0$ such that $g = g' g_0$, and hence
%    \begin{equation*}
%      m \rightsemiaction g G_0 = g_{m_0, m} g \leftaction m_0
%                          = g_{m_0, m} g' g_0 \leftaction m_0
%                          = g_{m_0, m} g' \leftaction m_0
%                          = m \rightsemiaction g' G_0.
%    \end{equation*}
    For each $m \in M$,
    \begin{equation*}
      m \rightsemiaction G_0 = m \rightsemiaction e_G G_0
                        = g_{m_0, m} e_G \leftaction m_0
                        = g_{m_0, m} \leftaction m_0
                        = m.
    \end{equation*}
    Let $H$ be a subgroup of $G$ such that $\set{g_{m_0, m} \suchthat m \in M} \subseteq H$. Furthermore, let $m \in M$ and let $h \in H$. Put $h_0 = g_{m_0, g_{m_0, m} h \leftaction m_0}^{-1} g_{m_0, m} h$. Because $g_{m_0, g_{m_0, m} h \leftaction m_0}^{-1} \in H$, $g_{m_0, m} \in H$, and
    \begin{align*}
      h_0 \leftaction m_0
      &= g_{m_0, g_{m_0, m} h \leftaction m_0}^{-1} \leftaction (g_{m_0, m} h \leftaction m_0)\\
      &= m_0,
    \end{align*}
    we have $h_0 \in H_0$. Moreover, for each $g G_0 \in G \quotient G_0$,
    \begin{align*}
      m \rightsemiaction h \cdot g G_0
      &= m \rightsemiaction h g G_0\\
      &= g_{m_0, m} h g \leftaction m_0\\
      &= g_{m_0, g_{m_0, m} h \leftaction m_0} h_0 g \leftaction m_0\\
      &= (g_{m_0, m} h \leftaction m_0) \rightsemiaction h_0 g G_0\\
      &= (m \rightsemiaction h G_0) \rightsemiaction h_0 \cdot g G_0. \tag*{\qed}
    \end{align*}
  \end{proof}

  \begin{example}
  \label{ex:group:liberation}
    In the situation of Example~\ref{ex:group:cell-space}, the stabiliser $G_0$ of the neutral element $m_0$ under $\cdot$ is the trivial subgroup $\set{e_G}$ of $G$ and, for each element $g \in G$, we have $g_{m_0, m} g g_{m_0, m}^{-1} m = m g m^{-1} m = m g$. Under the natural identification of $G \quotient G_0$ with $G$, the induced semi-action $\rightsemiaction$ is the right group action of $G$ on itself by multiplication.
  \end{example}

  \begin{example}
  \label{ex:sphere:liberation}
    In the situation of Example~\ref{ex:sphere:cell-space}, the stabiliser $G_0$ of the north pole $m_0$ under $\leftaction$ is the group of rotations about the $z$-axis. An element $g G_0 \in G \quotient G_0$ semi-acts on a point $m$ on the right by the induced semi-action $\rightsemiaction$ by first rotating $m$ to $m_0$, $g_{m_0, m}^{-1} \leftaction m = m_0$, secondly rotating $m_0$ as prescribed by $g$, $g g_{m_0, m}^{-1} \leftaction m = g \leftaction m_0$, and thirdly undoing the first rotation, $g_{m_0, m} g g_{m_0, m}^{-1} \leftaction m = g_{m_0, m} \leftaction (g \leftaction m_0)$, in other words, by first changing the rotation axis of $g$ such that the new axis stands to the line through the origin and $m$ as the old one stood to the line through the origin and $m_0$, $g_{m_0, m} g g_{m_0, m}^{-1}$, and secondly rotating $m$ as prescribed by this new rotation. % http://ell.stackexchange.com/questions/886/firstly-secondly-or-first-second

    Let $N_0$ be a subset of the sphere $M$, which we think of as a geometrical object on the sphere that has its centre at $m_0$, for example, a circle of latitude. The set $N = \set{g G_0 \in G \quotient G_0 \suchthat g \leftaction m_0 \in N_0} = \set{G_{m_0, m} \suchthat m \in N_0} = \set{g_{m_0, m} G_0 \suchthat m \in N_0}$ can be thought of as a realisation of $N_0$ in $G \quotient G_0$. Indeed, $m_0 \rightsemiaction N = g_{m_0, m_0} \set{g_{m_0, m} \suchthat m \in N_0} \leftaction m_0 = N_0$. Furthermore, for each point $m \in M$, the set $m \rightsemiaction N = g_{m_0, m} \leftaction N_0$ has the same shape and size as $N_0$ but its centre at $m$.
  \end{example}

%  \begin{lemma} % TADA Wird nicht benötigt. % TADA Einleitender Text ... einen Namen geben ... etwas mit "inverse"?
%  \label{lem:rightsemiaction-can-be-undone}
%%    \begin{equation*}
%%      \ForEach m \in M \ForEach g \in G \Exists g_0 \in G_0 \SuchThat \ForEach \mathfrak{g}' \in G \quotient G_0 \Holds
%%            m \rightsemiaction g g_0 \cdot \mathfrak{g}' = (m \rightsemiaction g G_0) \rightsemiaction \mathfrak{g}'
%%    \end{equation*}
%    \begin{equation*}
%      \ForEach m \in M \ForEach \mathfrak{g} \in G \quotient G_0 \Exists g \in \mathfrak{g} \SuchThat \ForEach \mathfrak{g}' \in G \quotient G_0 \Holds
%            m \rightsemiaction g \cdot \mathfrak{g}' = (m \rightsemiaction \mathfrak{g}) \rightsemiaction \mathfrak{g}',
%    \end{equation*}
%    in particular,
%    \begin{equation*}
%      \ForEach m \in M \ForEach \mathfrak{g} \in G \quotient G_0 \Exists g \in \mathfrak{g} \SuchThat
%          (m \rightsemiaction \mathfrak{g}) \rightsemiaction g^{-1} G_0 = m.
%    \end{equation*}
%%    \begin{equation*}
%%      \ForEach m \in M \ForEach \mathfrak{g} \in G \quotient G_0 \Exists g \in \mathfrak{g} \SuchThat \ForEach \mathfrak{g}' \in G \quotient G_0 \Holds
%%          (m \rightsemiaction \mathfrak{g}) \rightsemiaction g^{-1} \cdot \mathfrak{g}' = m \rightsemiaction \mathfrak{g}',
%%    \end{equation*}
%  \end{lemma}

  \begin{lemma}
  \label{lem:liberation:transitive-and-free}
    The semi-action $\rightsemiaction$ is
    \begin{enumerate}
      \item \label{it:liberation:transitive-and-free:transitive}
            \define{transitive}, which means that the set $M$ is non-empty and
            \begin{equation*}
              \ForEach m \in M \ForEach m' \in M \Exists \mathfrak{g} \in G \quotient G_0 \SuchThat m \rightsemiaction \mathfrak{g} = m';
            \end{equation*}
      \item \label{it:liberation:transitive-and-free:free}
            \define{free}, which means that
            \begin{equation*}
              \ForEach \mathfrak{g} \in G \quotient G_0 \ForEach \mathfrak{g}' \in G \quotient G_0 \Holds
                    (\Exists m \in M \SuchThat m \rightsemiaction \mathfrak{g} = m \rightsemiaction \mathfrak{g}') \implies \mathfrak{g} = \mathfrak{g}'.
            \end{equation*}
    \end{enumerate}
  \end{lemma}

  \begin{proof}%[Lemma~\ref{lem:liberation:transitive-and-free}]
    \begin{enumerate}
      \item Let $m \in M$ and let $m' \in M$. Put $m'' = g_{m_0, m}^{-1} \leftaction m'$. Because $\leftaction$ is transitive, there is a $g \in G$ such that $g \leftaction m_0 = m''$. Hence,
            \begin{align*}
              m \rightsemiaction g G_0 &= g_{m_0, m} g \leftaction m_0\\
                                  &= g_{m_0, m} \leftaction (g \leftaction m_0)\\
                                  &= g_{m_0, m} \leftaction m''\\
                                  &= g_{m_0, m} \leftaction (g_{m_0, m}^{-1} \leftaction m')\\
                                  &= e_G \leftaction m'\\
                                  &= m'.
            \end{align*}
      \item Let $g G_0$ and $g' G_0$ be two elements of $G \quotient G_0$, and let $m$ be an element of $M$ such that $m \rightsemiaction g G_0 = m \rightsemiaction g' G_0$. Then, $g_{m_0, m} g \leftaction m_0 = g_{m_0, m} g' \leftaction m_0$. Hence, $g \leftaction m_0 = g' \leftaction m_0$. Therefore, $g^{-1} g' \leftaction m_0 = m_0$. Thus, $g^{-1} g' \in G_0$. In conclusion, $g G_0 = g' G_0$. \qed
    \end{enumerate}
  \end{proof}

  \begin{lemma}
  \label{lem:semi-commutivity-of-liberation} % let $\rightsemiaction$ be the right quotient set semi-action of $G \quotient G_0$ on $M$ induced by $\leftaction$ and $\mathcal{K}$
    The semi-action $\rightsemiaction$
    \begin{enumerate}
      \item \label{it:semi-commutivity-of-liberation:semi-commutes}
            \define{semi-commutes with $\leftaction$}, which means that, for each subgroup $H$ of $G$ such that $\set{g_{m_0, m} \suchthat m \in M} \subseteq H$,
            \begin{equation*}
              \ForEach m \in M \ForEach h \in H \Exists h_0 \in H_0 \SuchThat \ForEach \mathfrak{g} \in G \quotient G_0 \Holds
                    (h \leftaction m) \rightsemiaction \mathfrak{g} = h \leftaction (m \rightsemiaction h_0 \cdot \mathfrak{g});
            \end{equation*}
      \item \label{it:semi-commutivity-of-liberation:exhausts}
            \define{exhausts its defect with respect to its semi-commutativity with $\leftaction$ in $m_0$}, which means that, for each subgroup $H$ of $G$ such that $\set{g_{m_0, m} \suchthat m \in M} \subseteq H$,
            \begin{equation*}
              \ForEach h_0 \in H_0 \ForEach \mathfrak{g} \in G \quotient G_0 \Holds
                    (h_0^{-1} \leftaction m_0) \rightsemiaction \mathfrak{g} = h_0^{-1} \leftaction (m_0 \rightsemiaction h_0 \cdot \mathfrak{g}).
            \end{equation*}
    \end{enumerate}
  \end{lemma}

  \begin{proof}%[Lemma~\ref{lem:semi-commutivity-of-liberation}]
    Let $H$ be a subgroup of $G$ such that $\set{g_{m_0, m} \suchthat m \in M} \subseteq H$.
    \begin{enumerate}
      \item Let $h \in H$ and let $m \in M$. Put $h_0 = g_{m_0, m}^{-1} h^{-1} g_{m_0, h \leftaction m}$. Because $g_{m_0, m}^{-1} \in H$, $g_{m_0, h \leftaction m}^{-1} \in H$, and
            \begin{align*}
              g_{m_0, m}^{-1} h^{-1} g_{m_0, h \leftaction m} \leftaction m_0 &= g_{m_0, m}^{-1} h^{-1} \leftaction (h \leftaction m)\\
                                                                              &= g_{m_0, m}^{-1} \leftaction m\\
                                                                              &= m_0,
            \end{align*}
            we have $h_0 \in H_0$. Moreover, for each $g G_0 \in G \quotient G_0$,
            \begin{align*}
              (h \leftaction m) \rightsemiaction g G_0
              &= g_{m_0, h \leftaction m} g \leftaction m_0\\
              &= h g_{m_0, m} h_0 g \leftaction m_0\\
              &= h \leftaction (g_{m_0, m} h_0 g \leftaction m_0)\\
              &= h \leftaction (m \rightsemiaction h_0 \cdot g G_0).
            \end{align*}
      \item For each $h_0 \in H_0$ and each $g G_0 \in G \quotient G_0$, because $g_{m_0, m_0} = e_G$,
            \begin{align*} % TADA Das sollte einfacher zu beweisen sein.
              (h_0^{-1} \leftaction m_0) \rightsemiaction g G_0
              &= m_0 \rightsemiaction g G_0\\
              &= g_{m_0, m_0} g \leftaction m_0\\
              &= h_0^{-1} g_{m_0, m_0} h_0 g \leftaction m_0\\
              &= h_0^{-1} \leftaction (g_{m_0, m_0} h_0 g \leftaction m_0)\\
              &= h_0^{-1} \leftaction (m_0 \rightsemiaction h_0 g G_0)\\
              &= h_0^{-1} \leftaction (m_0 \rightsemiaction h_0 \cdot g G_0). \tag*{\qed}
            \end{align*}
    \end{enumerate}
  \end{proof}

  \begin{lemma}
  \label{lem:identification-of-G-quotient-Gzero-with-M-by-right-semiaction}
    The maps
    \begin{equation*}
      \left\{
      \begin{aligned}
        m_0 \rightsemiaction \blank \from G \quotient G_0 &\to     M,\\
                                        \mathfrak{g} &\mapsto m_0 \rightsemiaction \mathfrak{g},
      \end{aligned}
      \right\}
      \text{ and }
      \left\{
      \begin{aligned}
        \iota \from M &\to     G \quotient G_0,\\
                    m &\mapsto G_{m_0, m},
      \end{aligned}
      \right\}
    \end{equation*}
    are inverse to each other and, under the identification of $G \quotient G_0$ with $M$ by either of these maps,
    \begin{equation*}
      \ForEach m \in M \ForEach \mathfrak{g} \in G \quotient G_0 \simeq M \Holds m \rightsemiaction \mathfrak{g} = g_{m_0, m} \leftaction \mathfrak{g}.
    \end{equation*}
%    $m_0 \rightsemiaction g G_0 = g \leftaction m_0$
%    $m \rightsemiaction g G_0 = g_{m_0, m} \leftaction (m_0 \rightsemiaction g G_0)$
  \end{lemma}

  \begin{proof}%[Lemma~\ref{lem:identification-of-G-quotient-Gzero-with-M-by-right-semiaction}]
    According to Lemma~\ref{lem:iota-is-bijective-and-equivariant}, the map $\iota$ is bijective and, for each $m \in M$,
    \begin{align*}
      m_0 \rightsemiaction \iota(m)
      &= m_0 \rightsemiaction G_{m_0, m}\\
      &= m_0 \rightsemiaction g_{m_0, m} G_0\\
      &= g_{m_0, m_0} g_{m_0, m} \leftaction m_0\\
      &= e_G \leftaction m\\
      &= m.
    \end{align*}
    Therefore, $m_0 \rightsemiaction \blank = \iota^{-1}$. Moreover, for each $g G_0 \in G \quotient G_0$,
    \begin{align*}
      m \rightsemiaction g G_0
      &= g_{m_0, m} g \leftaction m_0\\
      &= g_{m_0, m} \leftaction (g \leftaction m_0)\\
      &= g_{m_0, m} \leftaction \iota^{-1}(G_{m_0, g \leftaction m_0})\\
      &= g_{m_0, m} \leftaction \iota^{-1}(g G_0). \tag*{\qed}
    \end{align*}
  \end{proof}

  %----------------------------------------------------------------------------------------
  \section{Semi-Cellular and Cellular Automata}
  \label{sec:automata}
  %----------------------------------------------------------------------------------------

  In this section, let $\mathcal{R} = \ntuple{\mathcal{M}, \mathcal{K}} = \ntuple{\ntuple{M, G, \leftaction}, \ntuple{m_0, \family{g_{m_0, m}}_{m \in M}}}$ be a cell space.

  \begin{definition}
  \label{def:semi-cellular-automaton}
    Let $Q$ be a set, let $N$ be a subset of $G \quotient G_0$ such that $G_0 \cdot N \subseteq N$, and let $\delta$ be a map from $Q^N$ to $Q$. The quadruple $\mathcal{C} = \ntuple{\mathcal{R}, Q, N, \delta}$ is called \define{semi-cellular automaton}, each element $q \in Q$ is called \define{state}, the set $N$ is called \define{neighbourhood}, each element $n \in N$ is called \define{neighbour}, and the map $\delta$ is called \define{local transition function}.
  \end{definition}

  \begin{example}
  \label{ex:group:semi-cellular-automaton}
    In the situation of Example~\ref{ex:group:liberation}, the semi-cellular automata over $\mathcal{R}$ are the usual cellular automata over the group $G$. % TADA Auf Monographie verweisen
  \end{example}

  \begin{example}
  \label{ex:sphere:semi-cellular-automaton}
    In the situation of Example~\ref{ex:sphere:liberation}, let $Q$ be the set $\set{0,1}$, let $N_0$ be the union of all circles of latitude between $45\degree$ and $90\degree$ north, which is a curved circular disk of radius $\pi/4$ with the north pole $m_0$ at its centre, let $N$ be the set $\set{g G_0 \suchthat g \leftaction m_0 \in N_0}$, and let
    \begin{equation*} % TADA align* ...
      \delta \from Q^N \to     Q,
                  \ell \mapsto \begin{dcases*}
                                  0, &if $\forall n \in N \Holds \ell(n) = 0$,\\
                                  1, &if $\exists n \in N \SuchThat \ell(n) = 1$.
                                \end{dcases*}
    \end{equation*}
    The quadruple $\mathcal{C} = \ntuple{\mathcal{R}, Q, N, \delta}$ is a semi-cellular automaton.
  \end{example}

  In the remainder of this section, let $\mathcal{C} = \ntuple{\mathcal{R}, Q, N, \delta}$ be a semi-cellular automaton.

%  There is a natural bijection from $G \quotient G_0$ to $M$, namely $m_0 \rightsemiaction \blank$. $m \mapsto G_{m_0,m} (= g_{m_0, m} G_0 = g_{m_0, m} g_{m_0,m_0}^{-1} G_0)$. So, instead of having chosen $N$ as a subset of $G \quotient G_0$, we could just as well have chosen it as a subset of $M$.
%
%  So, instead of choosing the neighbourhood $N$ as a subset of $G \quotient G_0$ we can just as well choose it as a subset of $M$.
%  $m_0 \rightsemiaction g G_0 = g_{m_0,m_0}^{-1} g \leftaction m_0$

%  Suppose that $g_{m_0,m_0} = e_G$ and let $n = g G_0 \in N$. The corresponding neighbour cell of $m_0$ is $m_n = m_0 \rightsemiaction n = g \leftaction m_0$. The corresponding neighbour cell of $m$ is $m \rightsemiaction n = g_{m_0, m} \leftaction m_n$, that is, $m_n$ transported from the neighbourhood of $m_0$ to the one of $m$. If instead we had used $g \leftaction m$, strange things would have happend ....

%  \begin{definition}
%    The group $G$ acts on the quotient set $G \quotient G_0$ on the left by
%    \begin{align*}
%      \cdot \from G \times G \quotient G_0 &\to     G \quotient G_0,\\
%                               (g, g' G_0) &\mapsto g g' G_0.
%    \end{align*}
%  \end{definition}

  \begin{definition} % TADA "essential" klingt nach "minimal", aber tatsächlich tut es einfach N selbst!
    Let $E$ be a subset of $N$. It is called \define{essential neighbourhood} if and only if
    \begin{equation*}
      \ForEach \ell \in Q^N \ForEach \ell' \in Q^N \Holds \ell\restriction_E = \ell'\restriction_E \implies \delta(\ell) = \delta(\ell').
    \end{equation*}
  \end{definition}

  \begin{definition}
  \label{def:local-configuration}
    Each map $\ell \in Q^N$ is called \define{local configuration}. The stabiliser $G_0$ acts on $Q^N$ on the left by
    \begin{align*}
      \bullet \from G_0 \times Q^N &\to     Q^N,\\
                       (g_0, \ell) &\mapsto [n \mapsto \ell(g_0^{-1} \cdot n)].
    \end{align*}
  \end{definition}

  \begin{definition}
  \label{def:cellular-automaton}
    The semi-cellular automaton $\mathcal{C}$ is called \define{cellular automaton} if and only if its local transition function $\delta$ is $\bullet$-invariant.
  \end{definition}

  \begin{example}
  \label{ex:group:cellular-automaton}
    In the situation of Example~\ref{ex:group:semi-cellular-automaton}, the stabiliser $G_0$ of the neutral element $m_0$ is the trivial subgroup $\set{e_G}$ of $G$. Therefore, for each semi-cellular automaton over $\mathcal{R}$, its local transition function is $\bullet$-invariant and hence it is a cellular automaton.
  \end{example}

  \begin{example}
  \label{ex:sphere:cellular-automaton}
    In the situation of Example~\ref{ex:sphere:semi-cellular-automaton}, think of $0$ as black, $1$ as white, and of local configurations as black-and-white patterns on $N_0 = m_0 \rightsemiaction N$. The rotations $G_0$ about the $z$-axis act on these patterns by $\bullet$ by rotating them. The local transition function $\delta$ maps the black pattern to $0$ and all others to $1$, which is invariant under rotations. Therefore, the quadruple $\mathcal{C}$ is a cellular automaton.
  \end{example}

  \begin{definition}
  \label{def:global-configuration}
    Each map $c \in Q^M$ is called \define{global configuration}. The group $G$ acts on $Q^M$ on the left by
    \begin{align*}
      \inducedleftaction \from G \times Q^M &\to     Q^M,\\
                                      (g,c) &\mapsto [m \mapsto c(g^{-1} \leftaction m)].
    \end{align*}
  \end{definition}

  \begin{definition}
  \label{def:observed-local-configuration}
    For each global configuration $c \in Q^M$ and each cell $m \in M$, the local configuration
    \begin{align*}
      N &\to     Q,\\
      n &\mapsto c(m \rightsemiaction n),
    \end{align*}
    is called \define{observed by $m$ in $c$}.
  \end{definition}

  \begin{remark}
  \label{rem:local-configuration-induces-global-configuration}
    Because the semi-action $\rightsemiaction$ is free, for each local configuration $\ell \in Q^N$ and each cell $m \in M$, there is a global configuration $c \in Q^M$ such that the local configuration observed by $m$ in $c$ is $\ell$. % Because $\rightsemiaction$ is free, the cells $m \rightsemiaction n$, for $n \in N$, are pairwise distinct. % Item~\ref{\ref{it:liberation:transitive-and-free:free}} of Lemma~\ref{lem:liberation:transitive-and-free}
  \end{remark}

  \begin{definition}
  \label{def:global-transition-function}
    The map
    \begin{align*}
      \Delta \from Q^M &\to     Q^M,\\
                     c &\mapsto [m \mapsto \delta(n \mapsto c(m \rightsemiaction n))],
    \end{align*}
    is called \define{global transition function}.
  \end{definition}

  \begin{example}
    In the situation of Example~\ref{ex:sphere:cellular-automaton}, repeated applications of the global transition function of $\mathcal{C}$ grows white regions on $M$.
  \end{example}

  \begin{remark}
    For each subset $A$ of $M$ and each global configuration $c \in Q^M$, the states of the cells $A$ in $\Delta(c)$ depends at most on the states of the cells $A \rightsemiaction N$ in $c$. More precisely,
    \begin{equation*}
      \ForEach A \subseteq M \ForEach c \in Q^M \ForEach c' \in Q^M \Holds
          c\restriction_{A \rightsemiaction N} = c'\restriction_{A \rightsemiaction N} \implies \Delta(c)\restriction_A = \Delta(c')\restriction_A.
    \end{equation*}
  \end{remark}

  \begin{lemma}
  \label{lem:semi-commutativity-induced-left-action-and-bullet}
    Let $m$ be an element of $M$, let $g$ be an element of $G$, and let $g_0$ be an element of $G_0$ such that
    \begin{equation*}
      \ForEach n \in N \Holds (g^{-1} \leftaction m) \rightsemiaction n = g^{-1} \leftaction (m \rightsemiaction g_0 \cdot n).
    \end{equation*}
    For each global configuration $c \in Q^M$,
    \begin{equation*}
      [n \mapsto c((g^{-1} \leftaction m) \rightsemiaction n)] = g_0^{-1} \bullet [n \mapsto (g \inducedleftaction c)(m \rightsemiaction n)].
    \end{equation*}
  \end{lemma}

  \begin{proof}%[Lemma~\ref{lem:semi-commutativity-induced-left-action-and-bullet}]
    For each global configuration $c \in Q^M$,
    \begin{align*}
      [n \mapsto c((g^{-1} \leftaction m) \rightsemiaction n)]
      &= [n \mapsto c(g^{-1} \leftaction (m \rightsemiaction g_0 \cdot n))]\\
      &= g_0^{-1} \bullet [n \mapsto c(g^{-1} \leftaction (m \rightsemiaction n))]\\
      &= g_0^{-1} \bullet [n \mapsto (g \inducedleftaction c)(m \rightsemiaction n)]. \tag*{\qed}
    \end{align*}
  \end{proof}

  \begin{definition}
    The set $N_0 = m_0 \rightsemiaction N$ is called \define{neighbourhood of $m_0$}.
  \end{definition}

  \begin{definition}
    The map
    \begin{align*}
      \delta_0 \from Q^{N_0} &\to     Q,\\
                      \ell_0 &\mapsto \delta(n \mapsto \ell_0(m_0 \rightsemiaction n)),
    \end{align*}
    is called \define{local transition function of $m_0$}.
  \end{definition}

  \begin{lemma}
  \label{lem:global-transition-function-without-liberation}
    The global transition function $\Delta$ of $\mathcal{C}$ is identical to the map
    \begin{align*}
      \Delta_0 \from Q^M &\to     Q^M,\\
                       c &\mapsto [m \mapsto \delta_0((g_{m_0, m}^{-1} \inducedleftaction c)\restriction_{N_0})].
    \end{align*}
  \end{lemma}

  \begin{proof}%[Lemma~\ref{lem:global-transition-function-without-liberation}]
    Let $c \in Q^M$ and let $m \in M$. For each $n = g G_0 \in N$,
    \begin{align*} % TADA Diese elementare Eigenschaft auslagern?
      m \rightsemiaction n
      &= g_{m_0, m} g \leftaction m_0\\
      &= g_{m_0, m} g_{m_0, m_0} g \leftaction m_0\\
      &= g_{m_0, m} \leftaction (g_{m_0, m_0} g \leftaction m_0)\\
      &= g_{m_0, m} \leftaction (m_0 \rightsemiaction n)
    \end{align*}
    and thus
    \begin{equation*}
      c(m \rightsemiaction n) = c(g_{m_0, m} \leftaction (m_0 \rightsemiaction n)) = (g_{m_0, m}^{-1} \inducedleftaction c)(m_0 \rightsemiaction n).
    \end{equation*}
    Therefore,
    \begin{align*}
      \Delta(c)(m)
      &= \delta(n \mapsto c(m \rightsemiaction n))\\
      &= \delta(n \mapsto (g_{m_0, m}^{-1} \inducedleftaction c)(m_0 \rightsemiaction n))\\
      &= \delta_0(n_0 \mapsto (g_{m_0, m}^{-1} \inducedleftaction c)(n_0))\\
      &= \delta_0((g_{m_0, m}^{-1} \inducedleftaction c)\restriction_{N_0})\\
      &= \Delta_0(c)(m).
    \end{align*}
    In conclusion, $\Delta = \Delta_0$. \qed
  \end{proof}

  %----------------------------------------------------------------------------------------
  \section{Invariance, Equivariance, Determination, and Composition of Global Transition Functions}
  \label{sec:invariances}
  %----------------------------------------------------------------------------------------

  In Theorem~\ref{thm:independence-of-coordinate-system} we show that a global transition function does not depend on the choice of coordinate system. In Theorem~\ref{thm:local-invariance-versus-global-equivariance} we show that a global transition function is $\inducedleftaction$-equivariant if and only if the local transition function is $\bullet$-invariant. In Theorem~\ref{thm:determination-of-cellular-automata-by-behaviour-at-origin} we show that a global transition function is determined by its behaviour in the origin. And in Theorem~\ref{thm:composition-of-cellular-automata} we show that the composition of two global transition functions is a global transition function.

  \begin{lemma}
  \label{lem:left-group-action-on-union-of-quotients-by-stabilisers}
    Let $\leftaction$ be a left group action of $G$ on $M$. The group $G$ acts on $\bigcup_{m \in M} G \quotient G_m$ on the left by
    \begin{align*}
      \circ \from G \times \bigcup_{m \in M} G \quotient G_m &\to     \bigcup_{m \in M} G \quotient G_m,\\
                                                 (g, g' G_m) &\mapsto g g' G_m g^{-1}\ (= g g' g^{-1} G_{g \leftaction m}),
    \end{align*}
    such that, for each element $g \in G$ and each element $m \in M$, the map
    \begin{align*}
      (g \circ \blank)\restriction_{G \quotient G_m \to G \quotient G_{g \leftaction m}}
      \from G \quotient G_m &\to     G \quotient G_{g \leftaction m},\\
                     g' G_m &\mapsto g \circ g' G_m,
  %    (g \circ)\restriction_{G \quotient G_m \to G \quotient G_{g \leftaction m}}
    \end{align*}
    is bijective.
  \end{lemma}

  \begin{proof}%[Lemma~\ref{lem:left-group-action-on-union-of-quotients-by-stabilisers}]
    To see that the maps $\circ$ and $(g \circ \blank)\restriction_{G \quotient G_m \to G \quotient G_{g \leftaction m}}$ are well-defined, let $g \in G$, let $m \in M$, let $g' G_m \in G \quotient G_m$, and put $m' = g \leftaction m$. According to Lemma~\ref{lem:stabiliser-vs-transporter}, we have $G_{m'} = g G_m g^{-1}$. Therefore
    \begin{equation*}
      g g' G_m g^{-1} = g g' (g^{-1} g) G_m g^{-1}
                      = (g g' g^{-1}) (g G_m g^{-1})
                      = g g' g^{-1} G_{m'}
                      \in G \quotient G_{m'}.
    \end{equation*}
    To show that the map $\circ$ is a left group action, let $g G_m \in \bigcup_{m \in M} G \quotient G_m$. We have $e_G \circ g G_m = g G_m$. Moreover, for each $g' \in G$ and each $g'' \in G$,
    \begin{equation*}
      g' g'' \circ g G_m = g' g'' g G_m (g'')^{-1} (g')^{-1}
                         = g' \circ g'' g G_m (g'')^{-1}
                         = g' \circ (g'' \circ g G_m).
    \end{equation*}
    The map $(g \circ \blank)\restriction_{G \quotient G_m \to G \quotient G_{g \leftaction m}}$ is bijective, because its inverse is $(g^{-1} \circ \blank)\restriction_{G \quotient G_{g \leftaction m} \to G \quotient G_m}$. \qed
  \end{proof}

  \begin{lemma}
  \label{lem:liberation-and-coordinate-systems}
    Let $\mathcal{M} = \ntuple{M, G, \leftaction}$ be a left homogeneous space, let $\mathcal{K} = \ntuple{m_0, \family{g_{m_0, m}}_{m \in M}}$ and $\mathcal{K}' = \ntuple{m_0', \family{g_{m_0', m}'}_{m \in M}}$ be two coordinate systems for $\mathcal{M}$, let $H$ be a subgroup of $G$ such that $\set{g_{m_0, m} \suchthat m \in M} \cup \set{g_{m_0', m}' \suchthat m \in M} \subseteq H$, and let $h$ be an element of $H$ such that $h \leftaction m_0 = m_0'$. Then,
    \begin{equation*}
      \ForEach m \in M \Exists h_0 \in H_0 \SuchThat \ForEach \mathfrak{g}' \in G \quotient G_0' \Holds
          m \rightsemiaction' \mathfrak{g}' = m \rightsemiaction h_0 \cdot (h^{-1} \circ \mathfrak{g}').
    \end{equation*}
  \end{lemma}

  \begin{proof}%[Lemma~\ref{lem:liberation-and-coordinate-systems}]
    Let $m \in M$. Put $h_0 = g_{m_0, m}^{-1} g_{m_0', m}' h$. Then, $h_0 \in H_0$ and $g_{m_0', m}' = g_{m_0, m} h_0 h^{-1}$. Furthermore, let $g G_0' \in G \quotient G_0'$. Then,
    \begin{align*}
      m \rightsemiaction' g G_0' &= g_{m_0', m}' g (g_{m_0', m}')^{-1} \leftaction m\\
                            &= g_{m_0, m} h_0 h^{-1} g h h_0^{-1} g_{m_0, m}^{-1} \leftaction m\\
                            &= m \rightsemiaction h_0 h^{-1} g h h_0^{-1} G_0.
    \end{align*}
    Thus, because $h_0^{-1} G_0 = G_0$ and $h G_0 h^{-1} = h \circ G_0 = G_0'$,
    \begin{align*}
      m \rightsemiaction' g G_0' &= m \rightsemiaction h_0 \cdot h^{-1} g h G_0\\
                            &= m \rightsemiaction h_0 \cdot h^{-1} g h G_0 h^{-1} h\\
                            &= m \rightsemiaction h_0 \cdot h^{-1} g G_0' h\\
                            &= m \rightsemiaction h_0 \cdot (h^{-1} \circ g G_0'). \tag*{\qed}
    \end{align*}
  \end{proof}

  \begin{theorem}
  \label{thm:independence-of-coordinate-system}
    In the situation of Lemma~\ref{lem:liberation-and-coordinate-systems},
%    Let $\mathcal{M} = \ntuple{M, G, \leftaction}$ be a left homogeneous space, let $\mathcal{K} = \ntuple{m_0, \family{g_{m_0, m}}_{m \in M}}$ and $\mathcal{K}' = \ntuple{m_0', \family{g_{m_0', m}'}_{m \in M}}$ be two coordinate systems for $\mathcal{M}$, let $H$ be a subgroup of $G$ such that $\set{g_{m_0, m} \suchthat m \in M} \cup \set{g_{m_0', m}' \suchthat m \in M} \subseteq H$, and let $h$ be an element of $H$ such that $h \leftaction m_0 = m_0'$. Moreover,
    let $\mathcal{C} = \ntuple{\ntuple{\mathcal{M}, \mathcal{K}}, Q, N, \delta}$ be a semi-cellular automaton such that $\delta$ is $\bullet_{H_0}$-invariant, let $N'$ be the set $h \circ N$, and let
    \begin{align*}
      \delta' \from Q^{N'} &\to     Q,\\
                     \ell' &\mapsto \delta(n \mapsto \ell'(h \circ n)).
    \end{align*}
    The quadruple $\ntuple{\ntuple{\mathcal{M}, \mathcal{K}'}, Q, N', \delta'}$ is a semi-cellular automaton whose global transition function is identical to the one of $\mathcal{C}$. % TODO Mit dem nächsten Theorem sieht man das gilt: \delta' is $\bullet'\restriction_{H_0' \times Q^{N'}}$-invariant
  \end{theorem}

  \begin{proof}%[Theorem~\ref{thm:independence-of-coordinate-system}] % TADA Beweis für die erste Aussage
    We have $N' \subseteq G \quotient G_0'$ and $G_0' \cdot N' \subseteq N'$. Moreover, let $c \in Q^M$ and let $m \in M$. According to Lemma~\ref{lem:liberation-and-coordinate-systems}, there is an $h_0 \in H_0$ such that
    \begin{equation*}
      \ForEach n' \in N' \Holds m \rightsemiaction' n' = m \rightsemiaction h_0 \cdot (h^{-1} \circ n').
    \end{equation*}
    Therefore, because $\delta$ is $\bullet_{H_0}$-invariant,
    \begin{align*}
      \Delta'(c)(m) &= \delta'(n' \mapsto c(m \rightsemiaction' n'))\\
                    &= \delta'(n' \mapsto c(m \rightsemiaction h_0 \cdot (h^{-1} \circ n')))\\
                    &= \delta(n \mapsto c(m \rightsemiaction h_0 \cdot (h^{-1} \circ (h \circ n))))\\
                    &= \delta(n \mapsto c(m \rightsemiaction h_0 \cdot n))\\
                    &= \delta(h_0^{-1} \bullet [n \mapsto c(m \rightsemiaction n)])\\ %.
                    &= \delta(n \mapsto c(m \rightsemiaction n))\\
                    &= \Delta(c)(m).
    \end{align*}
%    Therefore, because $\delta$ is $\bullet_{H_0}$-invariant,
%    \begin{equation*}
%      \Delta'(c)(m) = \delta(n \mapsto c(m \rightsemiaction n))
%                    = \Delta(c)(m).
%    \end{equation*}
    In conclusion, $\Delta' = \Delta$. \qed
  \end{proof}

  \begin{corollary}
  \label{cor:independence-of-coordinate-system}
    Let $\ntuple{\mathcal{M}, \mathcal{K}} = \ntuple{\ntuple{M, G, \leftaction}, \ntuple{m_0, \family{g_{m_0, m}}_{m \in M}}}$ be a cell space and let $\mathcal{C} = \ntuple{\ntuple{\mathcal{M}, \mathcal{K}}, Q, N, \delta}$ be a cellular automaton. For each coordinate system $\mathcal{K}' = \ntuple{m_0', \family{g_{m_0', m}'}_{m \in M}}$ for $\mathcal{M}$, there is a cellular automaton $\ntuple{\ntuple{\mathcal{M}, \mathcal{K}'}, Q, N', \delta'}$ whose global transition function is identical to the one of $\mathcal{C}$. \qed
  \end{corollary}

  \begin{proof}
    This follows directly from Theorem~\ref{thm:independence-of-coordinate-system} with $H = G$ and $h = g_{m_0, m_0'}$. \qed
  \end{proof}

%  \begin{corollary} % TADA Auch für andere Theorem entsprechende Korollare einfügen?
%  \label{cor:independence-of-coordinate-system}
%    Let $\mathcal{M} = \ntuple{M, G, \leftaction}$ be a left homogeneous space, let $\mathcal{K} = \ntuple{m_0, \family{g_{m_0, m}}_{m \in M}}$ and $\mathcal{K}' = \ntuple{m_0', \family{g_{m_0', m}'}_{m \in M}}$ be two coordinate systems for $\mathcal{M}$ such that $m_0 = m_0'$, and let $\mathcal{C} = \ntuple{\ntuple{\mathcal{M}, \mathcal{K}}, Q, N, \delta}$ be a cellular automaton. The global transition function of the cellular automaton $\ntuple{\ntuple{\mathcal{M}, \mathcal{K}'}, Q, N, \delta}$ is identical to the one of $\mathcal{C}$.
%  \end{corollary}
%
%  \begin{proof}
%    This follows directly from Theorem~\ref{thm:independence-of-coordinate-system} with $H = G$ and $h = e_G$. \qed
%  \end{proof}

  \begin{theorem}
  \label{thm:local-invariance-versus-global-equivariance}
    Let $\mathcal{R} = \ntuple{\ntuple{M, G, \leftaction}, \ntuple{m_0, \family{g_{m_0, m}}_{m \in M}}}$ be a cell space, let $\mathcal{C} = \ntuple{\mathcal{R}, Q, N, \delta}$ be a semi-cellular automaton, and let $H$ be a subgroup of $G$ such that $\set{g_{m_0, m} \suchthat m \in M} \subseteq H$. % let $\Delta$ be the global transition function of $\mathcal{C}$
    \begin{enumerate}
      \item\label{it:local-invariance-versus-global-equivariance:local-to-global}
            If the local transition function $\delta$ is $\bullet_{H_0}$-invariant, then the global transition function $\Delta$ is $\inducedleftaction_H$-equivariant.
      \item\label{it:local-invariance-versus-global-equivariance:global-to-local}
            If there is an $\inducedleftaction_H$-equivariant map $\Delta_0 \from Q^M \to Q^M$ such that
            \begin{align}
            \label{eq:local-invariance-versus-global-equivariance:global-transition-function-at-origins}
              \ForEach c \in Q^M \Holds \Delta_0(c)(m_0) = \delta(n \mapsto c(m_0 \rightsemiaction n)),
            \end{align}
            then the local transition function $\delta$ is $\bullet_{H_0}$-invariant.
      \item\label{it:local-invariance-versus-global-equivariance:equivalence}
            The local transition function $\delta$ is $\bullet_{H_0}$-invariant if and only if the global transition function $\Delta$ is $\inducedleftaction_H$-equivariant.
    \end{enumerate}
  \end{theorem}

  \begin{proof}%[Theorem~\ref{thm:local-invariance-versus-global-equivariance}]
    \begin{enumerate}
      \item Let $\delta$ be $\bullet_{H_0}$-invariant. Furthermore, let $h \in H$, let $c \in Q^M$, and let $m \in M$. According to Item~\ref{it:semi-commutivity-of-liberation:semi-commutes} of Lemma~\ref{lem:semi-commutivity-of-liberation}, there is an $h_0 \in H_0$ such that
            \begin{equation*}
              \ForEach n \in N \Holds (h^{-1} \leftaction m) \rightsemiaction n = h^{-1} \leftaction (m \rightsemiaction h_0 \cdot n).
            \end{equation*}
            Hence, according to Lemma~\ref{lem:semi-commutativity-induced-left-action-and-bullet},
            \begin{align*}
              (h \inducedleftaction \Delta(c))(m) &= \Delta(c)(h^{-1} \leftaction m)\\
                                                  &= \delta(n \mapsto c((h^{-1} \leftaction m) \rightsemiaction n))\\
                                                  &= \delta(h_0^{-1} \bullet [n \mapsto (h \inducedleftaction c)(m \rightsemiaction n)])\\
                                                  &= \delta(n \mapsto (h \inducedleftaction c)(m \rightsemiaction n))\\
                                                  &= \Delta(h \inducedleftaction c)(m).
            \end{align*}
            In conclusion, $\Delta$ is $\inducedleftaction_H$-equivariant.
      \item Let $\Delta_0$ be an $\inducedleftaction_H$-equivariant map such that \eqref{eq:local-invariance-versus-global-equivariance:global-transition-function-at-origins} holds. Furthermore, let $\ell \in Q^N$ and let $h_0 \in H_0$. According to Remark~\ref{rem:local-configuration-induces-global-configuration}, there is a $c \in Q^M$ such that $[n \mapsto c(m_0 \rightsemiaction n)] = \ell$.
%            \begin{equation*}
%              \ForEach n \in N \Holds c(m_0 \rightsemiaction n) = \ell(n).
%            \end{equation*}
            According to Item~\ref{it:semi-commutivity-of-liberation:exhausts} of Lemma~\ref{lem:semi-commutivity-of-liberation},
            \begin{equation*}
              \ForEach n \in N \Holds (h_0^{-1} \leftaction m) \rightsemiaction n = h_0^{-1} \leftaction (m \rightsemiaction h_0 \cdot n).
            \end{equation*}
            Hence, according to Lemma~\ref{lem:semi-commutativity-induced-left-action-and-bullet},
            \begin{align*}
              \delta(h_0 \bullet \ell)
              &= \delta(h_0 \bullet [n \mapsto c(m_0 \rightsemiaction n)])\\
              &= \delta(h_0 \bullet [n \mapsto c((h_0^{-1} \leftaction m_0) \rightsemiaction n)])\\
              &= \delta(n \mapsto (h_0 \inducedleftaction c)(m_0 \rightsemiaction n))\\
              &= \Delta_0(h_0 \inducedleftaction c)(m_0).
            \end{align*}
            Because $\Delta_0$ is $\inducedleftaction_H$-equivariant,
            \begin{align*}
              \Delta_0(h_0 \inducedleftaction c)(m_0)
              &= (h_0 \inducedleftaction \Delta_0(c))(m_0)\\
              &= \Delta_0(c)(h_0^{-1} \leftaction m_0)\\
              &= \Delta_0(c)(m_0)\\
              &= \delta(\ell).
            \end{align*}
            Put the last two chains of equalities together to see that $\delta(h_0 \bullet \ell) = \delta(\ell)$. In conclusion, $\delta$ is $\bullet_{H_0}$-invariant.
      \item If $\delta$ is $\bullet_{H_0}$-invariant, then $\Delta$ is $\inducedleftaction_H$-equivariant by Item~\ref{it:local-invariance-versus-global-equivariance:local-to-global}. On the other hand, if $\Delta$ is $\inducedleftaction_H$-equivariant, then $\delta$ is $\bullet_{H_0}$-invariant by Item~\ref{it:local-invariance-versus-global-equivariance:global-to-local}. \qed
    \end{enumerate}
  \end{proof}

  \begin{corollary}
    Let $\mathcal{C}$ be a semi-cellular automaton. It is a cellular automaton if and only if its global transition function is $\inducedleftaction$-equivariant. \qed
  \end{corollary}

  \begin{proof}
    This follows directly from Item~\ref{it:local-invariance-versus-global-equivariance:equivalence} of Theorem~\ref{thm:local-invariance-versus-global-equivariance} with $H = G$. \qed
  \end{proof}

  \begin{theorem}
  \label{thm:determination-of-cellular-automata-by-behaviour-at-origin}
    Let $\mathcal{R} = \ntuple{\ntuple{M, G, \leftaction}, \ntuple{m_0, \family{g_{m_0, m}}_{m \in M}}}$ be a cell space, let $\mathcal{C} = \ntuple{\mathcal{R}, Q, N, \delta}$ be a semi-cellular automaton, let $\Delta_0$ be a map from $Q^M$ to $Q^M$, and let $H$ be a subgroup of $G$ such that $\set{g_{m_0, m} \suchthat m \in M} \subseteq H$. The following statements are equivalent:
    \begin{enumerate}
      \item The local transition function $\delta$ is $\bullet_{H_0}$-invariant and the global transition function of $\mathcal{C}$ is $\Delta_0$;
      \item The global transition function $\Delta_0$ is $\inducedleftaction_H$-equivariant and
            \begin{equation}
            \label{eq:global-transition-function-at-origins}
              \ForEach c \in Q^M \Holds \Delta_0(c)(m_0) = \delta(n \mapsto c(m_0 \rightsemiaction n)).
            \end{equation}
    \end{enumerate}
  \end{theorem}

  \begin{proof}%[Theorem~\ref{thm:determination-of-cellular-automata-by-behaviour-at-origin}]
    First, let $\delta$ be $\bullet_{H_0}$-invariant and let $\Delta_0$ be the global transition function of $\mathcal{C}$. According to Item~\ref{it:local-invariance-versus-global-equivariance:local-to-global} of Theorem~\ref{thm:local-invariance-versus-global-equivariance}, the map $\Delta_0$ is $\inducedleftaction_H$-equivariant and, according to Definition~\ref{def:global-transition-function}, \eqref{eq:global-transition-function-at-origins} holds.

    Secondly, let $\Delta_0$ be $\inducedleftaction_H$-equivariant and let \eqref{eq:global-transition-function-at-origins} hold. According to Item~\ref{it:local-invariance-versus-global-equivariance:global-to-local} of Theorem~\ref{thm:local-invariance-versus-global-equivariance}, the local transition function $\delta$ is $\bullet_{H_0}$-invariant. Furthermore, let $c \in Q^M$ and let $m \in M$. Put $h = g_{m_0, m}^{-1} \in H$. Then,
    \begin{equation*}
      \Delta_0(c)(m) = \Delta_0(c)(h^{-1} \leftaction m_0) = (h \inducedleftaction \Delta_0(c))(m_0).
    \end{equation*}
    Because $\Delta_0$ is $\inducedleftaction_H$-equivariant,
    \begin{equation*}
      (h \inducedleftaction \Delta_0(c))(m_0)
      = \Delta_0(h \inducedleftaction c)(m_0)
      = \delta(n \mapsto (h \inducedleftaction c)(m_0 \rightsemiaction n)). % According to \eqref{eq:global-transition-function-at-origins}
    \end{equation*}
    According to Lemma~\ref{lem:semi-commutivity-of-liberation}, there is an $h_0 \in H_0$ such that, for each $n \in N$, we have $(h^{-1} \leftaction m_0) \rightsemiaction n = h^{-1} \leftaction (m_0 \rightsemiaction h_0 \cdot n)$. Therefore, according to Lemma~\ref{lem:semi-commutativity-induced-left-action-and-bullet},
    \begin{align*}
      \delta(n \mapsto (h \inducedleftaction c)(m_0 \rightsemiaction n))
      &= \delta(h_0 \bullet [n \mapsto c((h^{-1} \leftaction m_0) \rightsemiaction n)])\\
      &= \delta(h_0 \bullet [n \mapsto c(m \rightsemiaction n)]).
    \end{align*}
    Because $\delta$ is $\bullet_{H_0}$-invariant,
    \begin{equation*}
      \delta(h_0 \bullet [n \mapsto c(m \rightsemiaction n)]) = \delta(n \mapsto c(m \rightsemiaction n)).
    \end{equation*}
    Put the last four chains of equalities together to see that
%    \begin{equation*}
      $\Delta_0(c)(m) = \delta(n \mapsto c(m \rightsemiaction n))$.
%    \end{equation*}
    In conclusion, $\Delta_0$ is the global transition function of $\mathcal{C}$. \qed
  \end{proof}

  \begin{theorem} % TADA Use respectively in the form: "x and y are P and Q respectively" (see http://english.stackexchange.com/a/24526)
  \label{thm:composition-of-cellular-automata}
    Let $\mathcal{R} = \ntuple{\ntuple{M, G, \leftaction}, \ntuple{m_0, \family{g_{m_0, m}}_{m \in M}}}$ be a cell space, let $\mathcal{C} = \ntuple{\mathcal{R}, Q, N, \delta}$ and $\mathcal{C}' = \ntuple{\mathcal{R}, Q, N', \delta'}$ be two semi-cellular automata, and let $H$ be a subgroup of $G$ such that $\set{g_{m_0, m} \suchthat m \in M} \subseteq H$, and $\delta$ and $\delta'$ are $\bullet_{H_0}$-invariant. Furthermore, let % TADA Irgendwie klarer machen, dass so ein H nicht existieren muss! % TADA "and the local transition functions ..." % let $\Delta$ and $\Delta'$ be the global transition functions of $\mathcal{C}$ and $\mathcal{C}'$ respectively
    \begin{equation*}
      N'' = \set{g \cdot n' \suchthat n \in N, n' \in N', g \in n}
    \end{equation*}
    and let
    \begin{align*}
      \delta'' \from Q^{N''} &\to     Q,\\
                      \ell'' &\mapsto \delta(n \mapsto \delta'(n' \mapsto \ell''(g_{m_0, m_0 \rightsemiaction n} \cdot n'))).
    \end{align*}
    The quadruple $\mathcal{C}'' = \ntuple{\mathcal{R}, Q, N'', \delta''}$ is a semi-cellular automaton whose local transition function is $\bullet_{H_0}$-invariant and whose global transition function is $\Delta \after \Delta'$.
  \end{theorem}

  \begin{proof}%[Theorem~\ref{thm:composition-of-cellular-automata}]
    The quadruple $\mathcal{C}'' = \ntuple{\mathcal{R}, Q, N'', \delta''}$ is a semi-cellular automaton. Because $\delta$ and $\delta'$ are $\bullet_{H_0}$-invariant, according to Item~\ref{it:local-invariance-versus-global-equivariance:local-to-global} of Theorem~\ref{thm:local-invariance-versus-global-equivariance}, the maps $\Delta$ and $\Delta'$ are $\inducedleftaction_H$-equivariant and hence $\Delta \after \Delta'$ also. Because $g_{m_0, m_0} = e_G$, % TADA g_{m_0, m_0 \rightsemiaction n} G_0 = n, because m_0 \rightsemiaction g_{m_0, m_0 \rightsemiaction n} G_0 = m_0 \rightsemiaction n and \rightsemiaction is free
    \begin{equation*} % TADA Vergleiche Eigenschaft von semi-actions % TADA ausführlichere Rechnung
      \ForEach n \in N \ForEach n' \in N' \Holds (m_0 \rightsemiaction n) \rightsemiaction n' = m_0 \rightsemiaction g_{m_0, m_0 \rightsemiaction n} \cdot n'.
    \end{equation*}
    Thus, for each $c \in Q^M$,
    \begin{align*}
      (\Delta \after \Delta')(c)(m_0) &= \delta(n \mapsto \Delta'(c)(m_0 \rightsemiaction n))\\
                                      &= \delta(n \mapsto \delta'(n' \mapsto c((m_0 \rightsemiaction n) \rightsemiaction n')))\\
                                      &= \delta(n \mapsto \delta'(n' \mapsto c(m_0 \rightsemiaction g_{m_0, m_0 \rightsemiaction n} \cdot n')))\\
                                      &= \delta''(n'' \mapsto c(m_0 \rightsemiaction n'')).
    \end{align*}
    Therefore, according to Theorem~\ref{thm:determination-of-cellular-automata-by-behaviour-at-origin}, the local transition function $\delta''$ is $\bullet_{H_0}$-invariant and the global transition function of $\mathcal{C}''$ is $\Delta \after \Delta'$. \qed
  \end{proof}

%  \begin{remark} % TADA? Das ist ein direkter Nachweis, dass \delta'' \bullet_{H_0}-invariant ist!
%    Let $h_0 \in H_0$. Then, for each $n \in N$,
%    \begin{equation*}
%      h_0^{-1} g_{m_0, m_0 \rightsemiaction n} G_0
%      = h_0^{-1} \cdot n
%      = g_{m_0, m_0 \rightsemiaction h_0^{-1} \cdot n} G_0 % because g_{m_0, m_0 \rightsemiaction x} G_0 = x (wie oben)
%    \end{equation*}
%    Hence, for each $n \in N$, there is an $h_{n,0} \in H_0$ such that $h_0^{-1} g_{m_0, m_0 \rightsemiaction n} = g_{m_0, m_0 \rightsemiaction h_0^{-1} \cdot n} h_{n,0}$. Therefore, because $\delta$ and $\delta'$ are $\bullet_{H_0}$-invariant, for each $\ell'' \in Q^{N''}$,
%    \begin{align*}
%      \delta''(h_0 \bullet \ell'')
%      &= \delta(n \mapsto \delta'(n' \mapsto (h_0 \bullet \ell'')(g_{m_0, m_0 \rightsemiaction n} \cdot n')))\\
%      &= \delta(n \mapsto \delta'(n' \mapsto \ell''(h_0^{-1} g_{m_0, m_0 \rightsemiaction n} \cdot n')))\\
%      &= \delta(n \mapsto \delta'(n' \mapsto \ell''(g_{m_0, m_0 \rightsemiaction h_0^{-1} \cdot n} \cdot (h_{n,0} \cdot n'))))\\
%      &= \delta(h_0 \bullet [n \mapsto \delta'(h_{n,0}^{-1} \bullet [n' \mapsto \ell''(g_{m_0, m_0 \rightsemiaction n} \cdot n')])])\\
%      &= \delta(n \mapsto \delta'(n' \mapsto \ell''(g_{m_0, m_0 \rightsemiaction n} \cdot n')))\\
%      &= \delta''(\ell'').
%    \end{align*}
%    In conclusion, $\delta''$ is $\bullet_{H_0}$-invariant.
%  \end{remark}

  %----------------------------------------------------------------------------------------
  \section{Curtis-Hedlund-Lyndon Theorems} % Characterisation of Cellular Automata
  \label{sec:characterisation}
  %----------------------------------------------------------------------------------------

  In this section we equip the the set of global configurations with a uniform and a topological structure, and prove a uniform and a topological variant of the Curtis-Hedlund-Lyndon theorem. In Main Theorem~\ref{thm:uniform-curtis-hedlund-lyndon}, the uniform variant, we show that global transition functions are characterised by $\inducedleftaction$-equivariance and uniform continuity. And in its Corollary~\ref{cor:curtis-hedlund-lyndon}, the topological variant, that under the assumption that the set of states is finite they are characterised by $\inducedleftaction$-equivariance and continuity.

  \begin{definition}
    Let $G$ be a group equipped with a topology. It is called \define{topological} if and only if the maps
    \begin{equation*}
      \left\{
      \begin{aligned}
        G \times G &\to     G,\\
            (g,g') &\mapsto g g',
      \end{aligned}
      \right\}
      \text{ and }
      \left\{
      \begin{aligned}
        G &\to     G,\\
        g &\mapsto g^{-1},
      \end{aligned}
      \right\}
    \end{equation*}
    are continuous, where $G \times G$ is equipped with the product topology.
  \end{definition}

  \begin{definition} % Confer https://en.wikipedia.org/wiki/Transversal_(combinatorics) % Note that we do not require the sets in $L$ to be pairwise disjoint and we do not require a one-to-one relation between the sets in $L$ and the elements of the transversal. This is the "less common" definition.
    Let $M$ be a set, let $\mathfrak{L}$ be a subset of the power set of $M$, and let $T$ be a subset of $M$. The set $T$ is called \define{transversal of $\mathfrak{L}$} if and only if there is a surjective map $f \from \mathfrak{L} \to T$ such that for each set $A \in \mathfrak{L}$ we have $f(A) \in A$.
  \end{definition}

  \begin{definition}
    Let $M$ and $M'$ be two topological spaces and let $f$ be a continuous map from $M$ to $M'$. The map $f$ is called
    \begin{enumerate}
      \item \define{proper} if and only if, for each compact subset $K$ of $M'$, its preimage $f^{-1}(K)$ is a compact subset of $M$; % under $f$
      \item \define{semi-proper} if and only if, for each compact subset $K$ of $M'$, each transversal of $\set{f^{-1}(k) \suchthat k \in K}$ is included in a compact subset of $M$.
    \end{enumerate}
  \end{definition}

  \begin{definition}
    Let $M$ be a topological space, let $G$ be a topological group, and let $\mathcal{M} = \ntuple{M, G, \leftaction}$ be a left group set. The group set $\mathcal{M}$ is called
    \begin{enumerate}
      \item \define{topological} if and only if the map $\leftaction$ is continuous;
      \item \define{proper} if and only if it is topological and the so-called \define{action map}
            \begin{align*}
              \alpha \from G \times M &\to     M \times M,\\
                               (g, m) &\mapsto (g \leftaction m, m),
            \end{align*}
            is proper, where $G \times M$ and $M \times M$ are equipped with their respective product topology;
      \item \define{semi-proper} if and only if it is topological and its action map is semi-proper.
    \end{enumerate}
  \end{definition}

  \begin{remark}
    Each proper map is semi-proper and each proper left group set is semi-proper.
  \end{remark}

  \begin{lemma}
  \label{lem:proper-compact-subset-of-group}
    Let $\mathcal{M} = \ntuple{M, G, \leftaction}$ be a semi-proper left group set, and let $K$ and $K'$ be two compact subsets of $M$. Each transversal of $\set{G_{k', k} \suchthat (k, k') \in K \times K'}$ is included in a compact subset of $G$.
%    \begin{enumerate}
%      \item If $\mathcal{M}$ is proper, then the set $\bigcup_{k \in K} \bigcup_{k' \in K'} G_{k', k}$ is a compact subset of $G$.
%      \item If $\mathcal{M}$ is semi-proper, then each transversal of $\set{G_{k', k} \suchthat k \in K, k' \in K'}$ is included in a compact subset of $M$.
%    \end{enumerate}
  \end{lemma}

  \begin{proof}%[Lemma~\ref{lem:proper-compact-subset-of-group}] % TADA In der ersten Aussage ist gemeint, dass es keinen Unterschied macht ob man erst die Teilraumtopologie nimmt und dann die Produkttopologie oder andersherum.
    According to Tychonoff's theorem and because product and subspace topologies behave well with each other, the set $K \times K'$ is a compact subset of $M \times M$. Because the action map $\alpha$ of $\mathcal{M}$ is semi-proper and the preimage of each tuple $(k, k') \in K \times K'$ under $\alpha$ is $G_{k', k} \times \set{k'}$, each transversal of
    \begin{equation*}
      \set{G_{k', k} \times \set{k'} \suchthat (k, k') \in K \times K'}
    \end{equation*}
    is included in a compact subset of $G \times M$. Because the canonical projection $\pi$ of $G \times M$ onto $G$ is continuous, each transversal of
    \begin{equation*} % TADA Ausführlicher.
      \set{G_{k', k} \suchthat (k, k') \in K \times K'}
    \end{equation*}
    is included in a compact subset of $G$. \qed
%    According to Tychonoff's theorem and because product and subspace topologies play nicely together, the set $K \times K'$ is a compact subset of $M \times M$. Because the action map $\alpha$ of $\leftaction$ is proper, the preimage
%    \begin{equation*}
%      \alpha^{-1}(K \times K') = \set{(g, k') \in G \times K' \suchthat g \leftaction k' \in K}
%    \end{equation*}
%    is compact. Because the canonical projection $\pi$ of $G \times M$ onto $G$ is continuous, the image
%    \begin{equation*}
%      \pi(\alpha^{-1}(K \times K')) = \set{g \in G \suchthat \Exists k' \in K' \SuchThat g \leftaction k' \in K}
%    \end{equation*}
%    is compact. That set is just
%    \begin{equation*}
%      \bigcup_{k \in K} \bigcup_{k' \in K'} \set{g \in G \suchthat g \leftaction k' = k}. \tag*{\qed}
%    \end{equation*}
  \end{proof}

%  \begin{corollary}
%  \label{cor:stabiliser-of-proper-compact}
%    Let $\ntuple{M, G, \leftaction}$ be a proper left group set and let $m$ be an element of $M$. The stabiliser $G_m$ of $m$ is a compact subset of $G$.
%  \end{corollary}
%
%  \begin{proof}
%    The singleton set $\set{m}$ is compact. Hence, according to Lemma~\ref{lem:proper-compact-subset-of-group}, the stabiliser $G_m = G_{m,m}$ is a compact subset of $G$. \qed
%  \end{proof}

  \begin{lemma}
  \label{lem:discrete-semi-proper-cell-space}
    Let $\mathcal{M} = \ntuple{M, G, \leftaction}$ be a left group set. Equip $M$ and $G$ with their respective discrete topology. The group set $\mathcal{M}$ is semi-proper.
%    Let $\mathcal{M} = \ntuple{M, G, \leftaction}$ be a left homogeneous space and let $m_0$ be an element of $M$ such that the stabiliser $G_0$ of $m_0$ under $\leftaction$ is finite. Equip $M$ and $G$ with their respective discrete topology. The group set $\mathcal{M}$ is proper.
  \end{lemma}

  \begin{proof}%[Lemma~\ref{lem:discrete-semi-proper-cell-space}]
    The product topology on $G \times M$ and the one on $M \times M$ are discrete. Hence, each subset of $G \times M$ and each of $M \times M$ is open. Thus, the map $\leftaction$ is continuous. Moreover, each subset of $G \times M$ and each of $M \times M$ is compact if and only if it is finite.

    Let $K$ be a finite subset of $M \times M$. Because the set $\set{\alpha^{-1}(k) \suchthat k \in K}$, where $\alpha$ is the action map of $\mathcal{M}$, is finite, so is each of its transversals. Thus, the map $\alpha$ is semi-proper. In conclusion, the group set $\mathcal{M}$ is semi-proper. \qed
%    % To show that $\alpha$ is proper, let
%    Let $K$ be a finite subset of $M \times M$ and let $K'$ be the preimage of $K$ under the action map $\alpha$ of $\leftaction$. For each tuple $(m', m) \in K$, there is an element $g \in G$ such that $g \leftaction m = m'$ and the transporter
%    \begin{equation*}
%      G_{m,m'} = g_{m,m'} G_m
%               = g_{m,m'} g_{m_0, m} G_0 g_{m_0, m}^{-1}
%    \end{equation*}
%    is finite and thus the preimage $\alpha^{-1}((m',m)) = G_{m,m'} \times \set{m}$ also. Therefore, the set $K' = \bigcup_{(m',m) \in K} \alpha^{-1}((m',m))$ is finite. Hence, the action map $\alpha$ is proper. In conclusion, the group set $\mathcal{M}$ is proper. \qed
  \end{proof}

  \begin{definition} % TADA Schöner formulieren und motivieren. % TADA "topological" ist für cell spaces gar nicht definiert
    Let $\mathcal{R} = \ntuple{\ntuple{M, G, \leftaction}, \ntuple{m_0, \family{g_{m_0, m}}_{m \in M}}}$ be a topological or uniform cell space. Equip $G \quotient G_0$ with the topology or uniformity induced by $m_0 \rightsemiaction \blank$. % TADA see lem:identification-of-G-quotient-Gzero-with-M-by-right-semiaction
  \end{definition}

  \begin{lemma}
  \label{lem:rightquotientaction-preserves-compacta} % Let $\mathcal{M} = \ntuple{M, G, \leftaction}$ be a semi-proper left homogeneous space, let $\mathcal{K} = \ntuple{m_0, \family{g_{m_0, m}}_{m \in M}}$ be a coordinate system for $\mathcal{M}$ % TADA semi-proper ist für "cell space" nicht definiert, aber für topologische left group sets ... selbiges Problem besteht auch in den folgenden Sätzen
    Let $\mathcal{R} = \ntuple{\ntuple{M, G, \leftaction}, \ntuple{m_0, \family{g_{m_0, m}}_{m \in M}}}$ be a semi-proper cell space, let $K$ be a compact subset of $M$, and let $E$ be a compact subset of $G \quotient G_0$. The set $K \rightsemiaction E$ is included in a compact subset of $M$.
  \end{lemma}

  \begin{proof}%[Lemma~\ref{lem:rightquotientaction-preserves-compacta}] % TADA Schöner formulieren.
    We have
    \begin{equation*}
      K \rightsemiaction E % TADA Ausführlicher. Die Eigenschaft $m \rightsemiaction \mathfrak{g} = g_{m_0, m} \leftaction (m_0 \rightsemiaction \mathfrak{g})$ in Lemma auslagern!
      = \bigcup_{k \in K} g_{m_0, k} \leftaction (m_0 \rightsemiaction E)
%      = (\bigcup_{k \in K} \set{g_{m_0, k}}) \leftaction (m_0 \rightsemiaction E)
      = \set{g_{m_0, k} \suchthat k \in K} \leftaction (m_0 \rightsemiaction E).
    \end{equation*}
    Both, the singleton set $\set{m_0}$ and the set $K$ are compact. Hence, according to Lemma~\ref{lem:proper-compact-subset-of-group}, the transversal $\set{g_{m_0, k} \suchthat k \in K}$ of $\set{G_{m_0, k} \suchthat (m_0, k) \in \set{m_0} \times K}$ is included in a compact subset of $G$. Thus, the product $\set{g_{m_0, k} \suchthat k \in K} \times (m_0 \rightsemiaction E)$ is included in a compact subset of $G \times M$. Therefore, because $\leftaction$ is continuous, the image $\set{g_{m_0, k} \suchthat k \in K} \leftaction (m_0 \rightsemiaction E)$ is included in a compact subset of $M$. \qed
%    According to Lemma~\ref{lem:stabiliser-vs-transporter} in the fourth step, we have
%    \begin{align*}
%      K \rightsemiaction N &= \bigcup_{k \in K} k \rightsemiaction G_0 \cdot N
%                      = \bigcup_{k \in K} k \rightsemiaction \set{g_0 h G_0 \suchthat g_0 \in G_0, h \in H}\\
%                      &= \bigcup_{k \in K} g_{k,m_0}^{-1} G_0 H \leftaction m_0
%                      = \bigcup_{k \in K} G_{m_0,k} H \leftaction m_0\\
%                      &= \bigcup_{k \in K} G_{m_0,k} \leftaction (H \leftaction m_0)
%                      = (\bigcup_{k \in K} G_{m_0,k}) \leftaction (H \leftaction m_0).
%    \end{align*}
%    The set $H$ is compact and so is the singleton set $\set{m_0}$. Hence, the product $H \times \set{m_0}$ is compact. Therefore, because $\leftaction$ is continuous, the image $H \leftaction m_0$ is compact. The set $K$ is compact and so is the singleton set $\set{m_0}$. Hence, according to Lemma~\ref{lem:proper-compact-subset-of-group}, the set $\bigcup_{k \in K} G_{m_0,k}$ is compact. Therefore, the product $(\bigcup_{k \in K} G_{m_0,k}) \times (H \leftaction m_0)$ is compact. Thus, because $\leftaction$ is continuous, the image $(\bigcup_{k \in K} G_{m_0,k}) \leftaction (H \leftaction m_0)$ is compact. \qed
  \end{proof}

  \begin{definition}
    Let $M$ be a topological space and let $Q$ be a set.
    \begin{enumerate}
      \item The topology on $Q^M$ that has for a subbase (and base) the sets
            \begin{equation*}
              \mathfrak{E}(K,b) = \set{c \in Q^M \suchthat c\restriction_K = b},
              \text{ for } b \in Q^K \text{ and } K \subseteq M \text{ compact},
            \end{equation*}
            is called \define{topology of discrete convergence on compacta}.
      \item The uniformity on $Q^M$ that has for a subbase (and base) the sets
            \begin{equation*}
              \mathfrak{E}(K) = \set{(c,c') \in Q^M \times Q^M \suchthat c\restriction_K = c'\restriction_K},
              \text{ for } K \subseteq M \text{ compact},
            \end{equation*}
            is called \define{uniformity of discrete convergence on compacta}.
    \end{enumerate}
  \end{definition}

  \begin{remark}
  \label{rem:topology-and-uniformity-of-discrete-convergence-on-compacta}
    \begin{enumerate}
      \item \label{it:topology-and-uniformity-of-discrete-convergence-on-compacta:prodiscrete-uniformity}
            Let $M$ be equipped with the discrete topology. For each subset $A$ of $M$, the set $A$ is compact if and only if it is finite. Therefore, the topology and uniformity of discrete convergence on compacta on $Q^M$ are the prodiscrete topology and uniformity on $Q^M$ respectively. For definitions of the latter see Sects.~1.2 and~1.9 in \cite{ceccherini-silberstein:coornaert:2010}. The set $Q^M$ equipped wit the prodiscrete topology is Hausdorff (see Proposition~1.2.1 in \cite{ceccherini-silberstein:coornaert:2010}) and if $Q$ is finite then it is compact (see the first paragraph in Sect.~1.8 in \cite{ceccherini-silberstein:coornaert:2010}).
      \item \label{it:compact-cell-space-uniformity-of-discrete-is-discrete}
            If the topological space $M$ is compact, the topology and uniformity of discrete convergence on compacta on $Q^M$ is the discrete topology and uniformity on $Q^M$ respectively.
      \item \label{it:topology-induced-by-uniformity-of-discrete-convergence-on-compacta}
            The topology induced by the uniformity of discrete convergence on compacta on $Q^M$ is the topology of discrete convergence on compacta on $Q^M$.
    \end{enumerate}
  \end{remark}

  \begin{remark} % TADA \Delta wurde als globale Überführungsfunktion eines ZA eingeführt, aber manchmal nennen wir auch Abbildungen so, die (noch) nicht von einem ZA her kommen. Das könnte irreführend sein. Selbiges gilt für \Delta_0.
  \label{rem:discrete-topological-cell-space}
    Let $\mathcal{R} = \ntuple{\ntuple{M, G, \leftaction}, \ntuple{m_0, \family{g_{m_0, m}}_{m \in M}}}$ be a cell space and let $Q$ be a finite set. Equip $M$ and $G$ with their respective discrete topologies, and equip $Q^M$ with the prodiscrete topology. According to Lemma~\ref{lem:discrete-semi-proper-cell-space}, the cell space $\mathcal{R}$ is semi-proper. Because the topology on $G \quotient G_0$ is discrete, each subset $E$ of $G \quotient G_0$ is compact if and only if it is finite. Because $Q^M$ is compact, each map $\Delta \from Q^M \to Q^M$ is uniformly continuous if and only if it is continuous. And, because $Q^M$ is Hausdorff, each map $\Delta \from Q^M \to Q^M$, is a uniform isomorphism if and only if it is continuous and bijective.
%    The prodiscrete topology on $Q^M$ is the product topology on $Q^M = \prod_{m \in M} Q$, where $Q$ is equipped with the discrete topology. Because $Q$ is finite, it is compact. Thus, according to Tychonoff's theorem, the topological space $Q^M$ is compact. Therefore, the map $\Delta$ is continuous if and only if it is uniformly continuous.
  \end{remark}

  \begin{main-theorem}[Uniform Variant; Morton Landers Curtis, Gustav Arnold Hedlund, and Roger Conant Lyndon, 1969]
  \label{thm:uniform-curtis-hedlund-lyndon}
    Let $\mathcal{R} = \ntuple{\ntuple{M, G, \leftaction}, \ntuple{m_0, \family{g_{m_0, m}}_{m \in M}}}$ be a semi-proper cell space, let $Q$ be a set, let $\Delta$ be a map from $Q^M$ to $Q^M$, let $Q^M$ be equipped with the uniformity of discrete convergence on compacta, and let $H$ be a subgroup of $G$ such that $\set{g_{m_0, m} \suchthat m \in M} \subseteq H$. The following statements are equivalent:
    \begin{enumerate}
      \item \label{it:uniform-curtis-hedlund-lyndon:global-transition-function}
            The map $\Delta$ is the global transition function of a semi-cellular automaton with $\bullet_{H_0}$-invariant local transition function and compact essential neighbourhood.
      \item \label{it:uniform-curtis-hedlund-lyndon:equivariant-and-continuous}
            The map $\Delta$ is $\inducedleftaction_H$-equivariant and uniformly continuous.
    \end{enumerate}
  \end{main-theorem}

  \begin{proof} % TADA "by Theorem ..." ---> "according to Theorem ..."
    First, let $\Delta$ be the global transition function of a semi-cellular automaton $\mathcal{C} = \ntuple{\mathcal{R}, Q, N, \delta}$ such that $\delta$ is $\bullet_{H_0}$-invariant and such that there is a compact essential neighbourhood $E$ of $\mathcal{C}$. Then, according to Item~\ref{it:local-invariance-versus-global-equivariance:equivalence} of Theorem~\ref{thm:local-invariance-versus-global-equivariance}, the map $\Delta$ is $\inducedleftaction_H$-equivariant. Moreover, let $K$ be a compact subset of $M$. According to Lemma~\ref{lem:rightquotientaction-preserves-compacta}, the set $K \rightsemiaction E$ is included in a compact subset $L$ of $M$. For each $c \in Q^M$ and each $c' \in Q^M$, if $c\restriction_{K \rightsemiaction E} = c'\restriction_{K \rightsemiaction E}$, then $\Delta(c)\restriction_K = \Delta(c')\restriction_K$, in particular, if $c\restriction_L = c'\restriction_L$, then $\Delta(c)\restriction_K = \Delta(c')\restriction_K$. Thus,
    \begin{equation*}
      (\Delta \times \Delta)(\mathfrak{E}(L)) \subseteq \mathfrak{E}(K).
    \end{equation*}
    Because the sets $\mathfrak{E}(K)$, for $K \subseteq M$ compact, constitute a base of the uniformity on $Q^M$, the global transition function $\Delta$ is uniformly continuous.

    Secondly, let $\Delta$ be as in Item~\ref{it:uniform-curtis-hedlund-lyndon:equivariant-and-continuous}. Because $\Delta$ is uniformly continuous, there is a compact subset $E_0$ of $M$ such that
    \begin{equation*}
      (\Delta \times \Delta)(\mathfrak{E}(E_0)) \subseteq \mathfrak{E}(\set{m_0}).
    \end{equation*}
    Therefore, for each $c \in Q^M$, the state $\Delta(c)(m_0)$ depends at most on $c\restriction_{E_0}$. The subset $E = (m_0 \rightsemiaction \blank)^{-1}(E_0) = \set{G_{m_0, m} \suchthat m \in E_0}$ of $G \quotient G_0$ is compact. Let $N$ be the set $G_0 \cdot E$. Then, $G_0 \cdot N \subseteq N$. And, because $E_0 \subseteq m_0 \rightsemiaction N$, for each $c \in Q^M$, the state $\Delta(c)(m_0)$ depends at most on $c\restriction_{m_0 \rightsemiaction N}$. Hence, there is a map $\delta \from Q^N \to Q$ such that
    \begin{equation*}
      \ForEach c \in Q^M \Holds \Delta(c)(m_0) = \delta(n \mapsto c(m_0 \rightsemiaction n)).
    \end{equation*}
    The quadruple $\mathcal{C} = \ntuple{\mathcal{R}, Q, N, \delta}$ is a semi-cellular automaton. Conclude with Theorem~\ref{thm:determination-of-cellular-automata-by-behaviour-at-origin} that $\delta$ is $\bullet_{H_0}$-invariant and that $\Delta$ is the global transition function of $\mathcal{C}$. \qed
  \end{proof}

  \begin{corollary}[Topological Variant; Morton Landers Curtis, Gustav Arnold Hedlund, and Roger Conant Lyndon, 1969]
  \label{cor:curtis-hedlund-lyndon}
    Let $\mathcal{R} = \ntuple{\ntuple{M, G, \leftaction}, \ntuple{m_0, \family{g_{m_0, m}}_{m \in M}}}$ be a cell space, let $Q$ be a finite set, let $\Delta$ be a map from $Q^M$ to $Q^M$, let $Q^M$ be equipped with the prodiscrete topology, and let $H$ be a subgroup of $G$ such that $\set{g_{m_0, m} \suchthat m \in M} \subseteq H$. The following statements are equivalent:
    \begin{enumerate}
      \item \label{it:curtis-hedlund-lyndon:global-transition-function}
            The map $\Delta$ is the global transition function of a semi-cellular automaton with $\bullet_{H_0}$-invariant local transition function and finite essential neighbourhood.
      \item \label{it:curtis-hedlund-lyndon:equivariant-and-continuous}
            The map $\Delta$ is $\inducedleftaction_H$-equivariant and continuous.
    \end{enumerate}
  \end{corollary}

  \begin{proof}
    With Remark~\ref{rem:discrete-topological-cell-space} this follows directly from Main Theorem~\ref{thm:uniform-curtis-hedlund-lyndon}. \qed
%    Equip $M$ and $G$ with their respective discrete topology, to make $\mathcal{M}$ topological (see Lemma~\ref{lem:discrete-semi-proper-cell-space}). The uniformity of discrete convergence on compacta on $Q^M$ is the prodiscrete uniformity on $Q^M$ and it induces the prodiscrete topology on $Q^M$ (see Item~\ref{it:topology-and-uniformity-of-discrete-convergence-on-compacta:prodiscrete-uniformity} of Remark~\ref{rem:topology-and-uniformity-of-discrete-convergence-on-compacta} and Item~\ref{it:topology-induced-by-uniformity-of-discrete-convergence-on-compacta} of Remark~\ref{rem:topology-and-uniformity-of-discrete-convergence-on-compacta}).
%
%    Because $M$ is equipped with the discrete topology, so is $G \quotient G_0$. Hence, for each subset $E$ of $G \quotient G_0$, the set $E$ is compact if and only if it is finite.
%
%    The prodiscrete topology on $Q^M$ is the product topology on $Q^M = \prod_{m \in M} Q$, where $Q$ is equipped with the discrete topology. Because $Q$ is finite, it is compact. Thus, according to Tychonoff's theorem, the topological space $Q^M$ is compact. Therefore, the map $\Delta$ is continuous if and only if it is uniformly continuous. \qed
  \end{proof}

  \begin{remark}
    In the case that $M = G$ and $\leftaction$ is the group multiplication of $G$, Main Theorem~\ref{thm:uniform-curtis-hedlund-lyndon} is Theorem~1.9.1 in \cite{ceccherini-silberstein:coornaert:2010} and Corollary~\ref{cor:curtis-hedlund-lyndon} is Theorem~1.8.1 in \cite{ceccherini-silberstein:coornaert:2010}.
  \end{remark}

  %----------------------------------------------------------------------------------------
  \section{Invertibility}
  \label{sec:invertibility}
  %----------------------------------------------------------------------------------------

  \begin{definition}
    Let $\mathcal{C} = \ntuple{\mathcal{R}, Q, N, \delta}$ be a semi-cellular automaton. It is called \define{invertible} if and only if there is a semi-cellular automaton $\mathcal{C}'$, called \define{inverse to $\mathcal{C}$}, such that the global transition functions of $\mathcal{C}$ and $\mathcal{C}'$ are inverse to each other.
  \end{definition}

  \begin{theorem}
  \label{thm:invertible-uniform-curtis-hedlund-lyndon}
    Let $\mathcal{R} = \ntuple{\ntuple{M, G, \leftaction}, \ntuple{m_0, \family{g_{m_0, m}}_{m \in M}}}$ be a semi-proper cell space, let $Q$ be a set, let $\Delta$ be a map from $Q^M$ to $Q^M$, let $Q^M$ be equipped with the uniformity of discrete convergence on compacta, and let $H$ be a subgroup of $G$ such that $\set{g_{m_0, m} \suchthat m \in M} \subseteq H$. The following statements are equivalent:
    \begin{enumerate}
      \item \label{it:invertible-curtis-hedlund-lyndon:automaton}
            The map $\Delta$ is the global transition function of an invertible semi-cellular automaton $\mathcal{C}$ that has an inverse $\mathcal{C}'$ such that the local transition functions of $\mathcal{C}$ and $\mathcal{C}'$ are $\bullet_{H_0}$-invariant, and $\mathcal{C}$ and $\mathcal{C}'$ have compact essential neighbourhoods.
      \item \label{it:invertible-curtis-hedlund-lyndon:equivariant}
            The map $\Delta$ is an $\inducedleftaction_H$-equivariant uniform isomorphism.
    \end{enumerate}
  \end{theorem}

  \begin{proof}
    With Lemma~\ref{lem:inverse-is-equivariant} this follows directly from Main Theorem~\ref{thm:uniform-curtis-hedlund-lyndon}.
%    First, let everything be as in Item~\ref{it:invertible-curtis-hedlund-lyndon:automaton}. Furthermore, let $\Delta'$ be the global transition function of $\mathcal{C}'$. According to Main Theorem~\ref{thm:uniform-curtis-hedlund-lyndon}, the maps $\Delta$ and $\Delta'$ are both $\inducedleftaction_H$-equivariant and uniformly continuous. Because $\mathcal{C}$ and $\mathcal{C}'$ are inverse to each other, the maps $\Delta$ and $\Delta'$ are also. In conclusion, the map $\Delta$ is an $\inducedleftaction_H$-equivariant uniform isomorphism.
%
%    Secondly, let everything be as in Item~\ref{it:invertible-curtis-hedlund-lyndon:equivariant}. According to Lemma~\ref{lem:inverse-is-equivariant}, the inverse $\Delta^{-1}$ is $\inducedleftaction_H$-equivariant. Moreover, it is uniformly continuous. According to Main Theorem~\ref{thm:uniform-curtis-hedlund-lyndon}, there are semi-cellular automata $\mathcal{C} = \ntuple{\mathcal{R}, Q, N, \delta}$ and $\mathcal{C}' = \ntuple{\mathcal{R}, Q, N', \delta'}$ such that $\Delta$ and $\Delta^{-1}$ are their respective global transition functions, $\delta$ and $\delta'$ are $\bullet_{H_0}$-invariant, and they have compact essential neighbourhoods. \qed
  \end{proof}

  \begin{corollary}
  \label{cor:invertible-curtis-hedlund-lyndon}
    Let $\mathcal{R} = \ntuple{\ntuple{M, G, \leftaction}, \ntuple{m_0, \family{g_{m_0, m}}_{m \in M}}}$ be a cell space, let $Q$ be a finite set, let $\Delta$ be a map from $Q^M$ to $Q^M$, let $Q^M$ be equipped with the prodiscrete topology, and let $H$ be a subgroup of $G$ such that $\set{g_{m_0, m} \suchthat m \in M} \subseteq H$. The following statements are equivalent:
    \begin{enumerate}
      \item \label{it:cor-invertible-curtis-hedlund-lyndon:automaton}
            The map $\Delta$ is the global transition function of an invertible semi-cellular automaton $\mathcal{C}$ that has an inverse $\mathcal{C}'$ such that the local transition functions of $\mathcal{C}$ and $\mathcal{C}'$ are $\bullet_{H_0}$-invariant, and $\mathcal{C}$ and $\mathcal{C}'$ have finite essential neighbourhoods.
      \item \label{it:cor-invertible-curtis-hedlund-lyndon:equivariant}
            The map $\Delta$ is $\inducedleftaction_H$-equivariant, continuous, and bijective.
    \end{enumerate}
  \end{corollary}

  \begin{proof}
    With Remark~\ref{rem:discrete-topological-cell-space} this follows directly from Theorem~\ref{thm:invertible-uniform-curtis-hedlund-lyndon}. \qed
%    ... Because $Q^M$ is compact and Hausdorff, the map $\Delta$ is bijective and continuous if and only if it is a uniform isomorphism.
  \end{proof}

%  \begin{proof}
%    According to \cref{TADA}, the inverse $\Delta^{-1}$ is $\inducedleftaction$-equivariant. According to \cref{lem:phase-space-is-compact}, the set $Q^M$ is compact. Hence, according to \cref{TADA}, the map $\Delta^{-1}$ is continuous. Therefore, according to \cref{thm:curtis-hedlund-lyndon}, there is a cellular automaton $\mathcal{C}' = \ntuple{\mathcal{M}, Q, N', \delta'}$ with global transition function $\Delta^{-1}$. \qed
%  \end{proof}

  \begin{corollary}
    Let $\mathcal{R} = \ntuple{\ntuple{M, G, \leftaction}, \ntuple{m_0, \family{g_{m_0, m}}_{m \in M}}}$ be a cell space, let $H$ be a subgroup of $G$ such that $\set{g_{m_0, m} \suchthat m \in M} \subseteq H$. Furthermore, let $\mathcal{C} = \ntuple{\mathcal{R}, Q, N, \delta}$ be a semi-cellular automaton with finite set of states, finite essential neighbourhood, and $\bullet_{H_0}$-invariant local transition function. The automaton $\mathcal{C}$ is invertible if and only if its global transition function is bijective.
  \end{corollary}

  \begin{proof}
    With Item~\ref{it:local-invariance-versus-global-equivariance:local-to-global} of Theorem~\ref{thm:local-invariance-versus-global-equivariance} and Corollary~\ref{cor:curtis-hedlund-lyndon} this follows directly from Corollary~\ref{cor:invertible-curtis-hedlund-lyndon}. \qed
  \end{proof}

  \newpage

  \appendix

%  In Appendix~\ref{apx:uniformities} we present the basic theory of uniformities.

%  \begin{proof}[Lemma~\ref{lem:rightsemiaction-can-be-undone}] % TADA Statement has changed
%    Let $m \in M$ and let $\mathfrak{g} \in G \quotient G_0$. There is an $h \in \mathfrak{g}$ such that $h G_0 = \mathfrak{g}$. And, there is a $g_0 \in G_0$ such that
%    \begin{equation*}
%      \ForEach \mathfrak{g}' \in G \quotient G_0 \Holds m \rightsemiaction h \cdot (h^{-1} \cdot \mathfrak{g}') = (m \rightsemiaction h G_0) \rightsemiaction g_0 \cdot (h^{-1} \cdot \mathfrak{g}').
%    \end{equation*}
%    Let $g = h g_0^{-1}$. Then, $g \in \mathfrak{g}$. And, because $h \cdot (h^{-1} \cdot \mathfrak{g}') = \mathfrak{g}'$, $h G_0 = \mathfrak{g}$, and $g_0 \cdot (h^{-1} \cdot \mathfrak{g}') = g^{-1} \cdot \mathfrak{g}'$,
%    \begin{equation*}
%      \ForEach \mathfrak{g}' \in G \quotient G_0 \Holds m \rightsemiaction \mathfrak{g}' = (m \rightsemiaction \mathfrak{g}) \rightsemiaction g^{-1} \cdot \mathfrak{g}'. \tag*{\qed}
%    \end{equation*}
%  \end{proof}

  %----------------------------------------------------------------------------------------
  \section{Uniformity}
  \label{apx:uniformities}
  %----------------------------------------------------------------------------------------

  The theory of uniformities as presented here may be found in more detail in Appendix~B in the monograph \enquote{Cellular Automata and Groups}\cite{ceccherini-silberstein:coornaert:2010}.

  \begin{definition}
    Let $X$ be a set. The diagonal $\set{(x,x) \suchthat x \in X}$ in $X \times X$ is denoted by $\Delta_X$.
  \end{definition}

  \begin{definition}
    Let $X$ be a set, let $R$ be a subset of $X \times X$, and let $y$ be an element of $X$. The $y$-th column $\set{x \in X \suchthat (x,y) \in R}$ in $R$ is denoted by $R[y]$.
  \end{definition}

  \begin{definition}
    Let $X$ be a set and let $R$ be a subset of $X \times X$. The inverse $\set{(y,x) \suchthat (x,y) \in R}$ of $R$ is denoted by $R^{-1}$. The set $R$ is called \define{symmetric} if and only if it is self-inverse, that is, $R^{-1} = R$.
  \end{definition}

  \begin{definition}
    Let $X$ be a set, and let $R$ and $R'$ be two subsets of $X \times X$. The composite $\set{(x,z) \suchthat \Exists y \in X \SuchThat (x,y) \in R \wedge (y,z) \in R'}$ is denoted by $R \circ R'$.
  \end{definition}

  \begin{definition}
    Let $X$ be a set and let $\mathcal{U}$ be a set of subsets of $X \times X$. The set $\mathcal{U}$ is called \define{uniformity on $X$} if and only if % the following statements hold:
    \begin{gather*}
      \mathcal{U} \neq \emptyset,\\
      \ForEach E \in \mathcal{U} \Holds \Delta_X \subseteq E,\\
      \ForEach E \in \mathcal{U} \ForEach E' \subseteq X \times X \Holds (E \subseteq E' \implies E' \in \mathcal{U}),\\
      \ForEach E \in \mathcal{U} \ForEach E' \in \mathcal{U} \Holds E \cap E' \in \mathcal{U},\\
      \ForEach E \in \mathcal{U} \Holds E^{-1} \in \mathcal{U},\\
      \ForEach E \in \mathcal{U} \Exists E' \in \mathcal{U} \SuchThat E' \circ E' \subseteq E.
    \end{gather*}
  \end{definition}

  \begin{definition}
    Let $X$ be a set and let $\mathcal{U}$ be a uniformity on $X$. The tuple $(X, \mathcal{U})$ is called \define{uniform space} and each element $E \in \mathcal{U}$ is called \define{entourage of $X$}.
  \end{definition}

  \begin{example}
    Let $X$ be a set and let $\mathcal{U}$ be the set of all subsets of $X \times X$ that contain $\Delta_X$. Then, $\mathcal{U}$ is the largest uniformity on $X$. Itself as well as the uniform space $(X, \mathcal{U})$ are called \define{uniformly discrete}. % TADA "uniformly discrete" war ein Vorschlag eines Rezensenten: "A note about Appendix B: at the end of Example 9, to avoid confusion with topological spaces, "are called discrete" should be "are called uniformly discrete". Actually (I.5 in "Uniform Spaces" by Isbell, AMS, 1964) there are uncountably many countable uniform spaces that are pairwise not uniformly equivalent (no u.c. bijection with u.c. inverse) but whose induced topology is discrete."
  \end{example}

  \begin{definition}
    Let $(X, \mathcal{U})$ be a uniform space. The set
    \begin{equation*}
      \set{O \subseteq X \suchthat \ForEach x \in O \Exists E \in \mathcal{U} \SuchThat E[x] = O}
    \end{equation*}
    is a topology on $X$ and called \define{induced by $\mathcal{U}$}.
  \end{definition}

  \begin{example}
    Let $(X, \mathcal{U})$ be a discrete uniform space. The topology induced by $\mathcal{U}$ is the discrete topology on $X$.
  \end{example}

  \begin{definition}
    Let $(X, \mathcal{U})$ be a uniform space and let $\mathcal{B}$ be a subset of $\mathcal{U}$.
    \begin{enumerate}
      \item The set $\mathcal{B}$ is called \define{base of $\mathcal{U}$} if and only if
            \begin{equation*}
              \ForEach E \in \mathcal{U} \Exists B \in \mathcal{B} \SuchThat B \subseteq E.
            \end{equation*}
      \item The set $\mathcal{B}$ is called \define{subbase of $\mathcal{U}$} if and only if the set
            \begin{equation*}
              \set{\bigcap_{i = 1}^n B_i \suchthat B_i \in \mathcal{B}, i \in \set{1,2,\dotsc,n}, n \in \naturals}
            \end{equation*}
            is a base of $\mathcal{U}$.
    \end{enumerate}
  \end{definition}

  \begin{lemma}
    Let $X$ be a set and let $\mathcal{B}$ be a set of subsets of $X \times X$.
    \begin{enumerate}
      \item The set $\mathcal{B}$ is a base of a uniformity on $X$ if and only if
            \begin{gather*}
              \mathcal{B} \neq \emptyset,\\
              \ForEach B \in \mathcal{B} \Holds \Delta_X \subseteq B,\\
              \ForEach B \in \mathcal{B} \ForEach B' \in \mathcal{B} \Exists B'' \in \mathcal{B} \SuchThat B'' \subseteq B \cap B',\\
              \ForEach B \in \mathcal{B} \Exists B' \in \mathcal{B} \SuchThat B' \subseteq B^{-1},\\
              \ForEach B \in \mathcal{B} \Exists B' \in \mathcal{B} \SuchThat B' \circ B' \subseteq B.
            \end{gather*}
      \item The set $\mathcal{B}$ is a subbase of a uniformity on $X$ if and only if
            \begin{gather*}
              \mathcal{B} \neq \emptyset,\\
              \ForEach B \in \mathcal{B} \Holds \Delta_X \subseteq B,\\
              \ForEach B \in \mathcal{B} \Exists B' \in \mathcal{B} \SuchThat B' \subseteq B^{-1},\\
              \ForEach B \in \mathcal{B} \Exists B' \in \mathcal{B} \SuchThat B' \circ B' \subseteq B.
            \end{gather*}
    \end{enumerate}
  \end{lemma}

  \begin{definition}
    Let $(X, \mathcal{U})$ and $(X', \mathcal{U}')$ be two uniform spaces and let $f$ be a map from $X$ to $X'$. The map $f$ is called \define{uniformly continuous} if and only if
    \begin{equation*}
      \ForEach E' \in \mathcal{U}' \Exists E \in \mathcal{U} \SuchThat (f \times f)(E) \subseteq E',
    \end{equation*}
    where
    \begin{align*}
      f \times f \from X \times X &\to     X' \times X',\\
                       (x_1, x_2) &\mapsto (f(x_1), f(x_2)).
    \end{align*}
  \end{definition}

  \begin{lemma}
    Let $(X, \mathcal{U})$ and $(X', \mathcal{U}')$ be two uniform spaces, let $f$ be a map from $X$ to $X'$, and let $\mathcal{B}'$ be a base or subbase of $\mathcal{U}'$. The map $f$ is uniformly continuous if and only if
    \begin{equation*}
      \ForEach E' \in \mathcal{B}' \Holds (f \times f)^{-1}(E') \in \mathcal{U}.
    \end{equation*}
  \end{lemma}

  \begin{corollary}
    Let $(X, \mathcal{U})$ and $(X', \mathcal{U}')$ be two uniform spaces, let $f$ be a map from $X$ to $X'$, and let $\mathcal{B}$ and $\mathcal{B}'$ be two bases or subbases of $\mathcal{U}$ and $\mathcal{U}'$ respectively. The map $f$ is uniformly continuous if and only if
    \begin{equation*}
      \ForEach E' \in \mathcal{B}' \Exists E \in \mathcal{B} \SuchThat (f \times f)(E) \subseteq E'.
    \end{equation*}
  \end{corollary} % TADA Proof (Das ist keine Aussage aus dem Buch)

  \begin{lemma}
    Let $(X, \mathcal{U})$ and $(X', \mathcal{U}')$ be two uniform spaces and let $f \from X \to X'$ be a uniformly continuous map. The map $f$ is continuous, where $X$ and $X'$ are equipped with the topologies induced by $\mathcal{U}$ and $\mathcal{U}'$ respectively.
  \end{lemma}

  \begin{theorem}
    Let $(X, \mathcal{U})$ and $(X', \mathcal{U}')$ be two uniform spaces such that $X$, equipped with the topology induced by $\mathcal{U}$, is compact. Each continuous map $f \from X \to X'$ is uniformly continuous.
  \end{theorem}

  \begin{definition}
    Let $(X, \mathcal{U})$ and $(X', \mathcal{U}')$ be two uniform spaces and let $f$ be a map from $X$ to $X'$. The map $f$ is called \define{uniform isomorphism} if and only if it is bijective, and both $f$ and $f^{-1}$ are uniformly continuous.
  \end{definition}

  \begin{lemma} % Confer B.2.5
    Let $(X, \mathcal{U})$ and $(X', \mathcal{U}')$ be two uniform spaces such that $X$, equipped with the topology induced by $\mathcal{U}$, is compact and $X'$, equipped with the topology induced by $\mathcal{U}'$, is Hausdorff. Each continuous and bijective map $f \from X \to X'$ is a uniform isomorphism.
  \end{lemma}

\end{document}